\documentclass[12pt]{amsart}
\usepackage{verbatim,amscd,amssymb,hyperref}
\usepackage[active]{srcltx}
\usepackage[all]{xy}
\usepackage{graphicx}

\newtheorem{thm}{Theorem}[section]
\newtheorem{prop}[thm]{Proposition}
\newtheorem{lem}[thm]{Lemma}
\newtheorem{cor}[thm]{Corollary}

\newtheorem*{Mthm}{Main Theorem}

\theoremstyle{definition}
\newtheorem{dfn}[thm]{Definition}

\def\C{\mathbb{C}}
\def\c{{\bf c}}
\def\D{\mathbb{D}}
\def\H{\mathbb{H}}
\def\R{\mathbb{R}}

\def\Q{\mathbb{Q}}
\def\0{\emptyset}
\def\Ac{\mathcal{A}}  \def\Cc{\mathcal{C}}

\def\Fc{\mathcal{F}}   \def\Mc{\mathcal{M}}
\def\Pc{\mathcal{P}} \def\Rc{\mathcal{R}}  
   \def\Wc{\mathcal{W}}
 \def\Zc{\mathcal{Z}}  

\def\Z{\mathbb{Z}}
\renewcommand\emptyset{\varnothing}
\newcommand{\sm}{\setminus}

\def\eps{\varepsilon}
\def\ol{\overline}
\def\d{\partial}
  \def\ta{\theta}  \def\Ga{\Gamma} 
    \def\la{\lambda} 
\def\om{\omega}

\def\vp{\varphi}
\def\le{\leqslant}
\def\ge{\geqslant}

\def\Im{{\rm Im}}

\def\uc{\mathbb{S}^1}
\def\bd{\mathrm{Bd}}
\def\disk{\mathbb{D}}
\def\<{\langle}
\def\>{\rangle}

\def\r{\mathsf{R}}
\def\CP{\mathbb{C}P}
\def\jac{\mathrm{Jac}}
\def\c{\mathsf{C}}
\def\rt{\mathrm{Rt}}

\begin{document}
%\date{June 28, 2021; revised ...}

\title{On critical renormalization of complex polynomials}

\author[A.~Blokh]{Alexander~Blokh}

\author[P.~Ha\"issinsky]{Peter Ha\"issinsky}

\author[L.~Oversteegen]{Lex Oversteegen}

\thanks{The third named author was partially  supported
by NSF grant DMS-1807558}

\author[V.~Timorin]{Vladlen~Timorin}

\thanks{The fourth named author has been supported by the HSE University Basic Research Program.}

\address[Alexander~Blokh and Lex~Oversteegen]
{Department of Mathematics\\ University of Alabama at Birmingham\\
Birmingham, AL 35294-1170}

\address[Peter~Ha\"issinsky]
{Aix-Marseille Univ., CNRS\\
I2M,  UMR 7373\\
39, rue Fr\'ed\'eric Joliot Curie\\
13453 Marseille Cedex 13, France}

\address[Vladlen~Timorin]
{Faculty of Mathematics\\
HSE University, Russian Federation\\
6 Usacheva St., 119048 Moscow
}

\address[Vladlen~Timorin]
{Independent University of Moscow\\
Bolshoy Vlasyevskiy Pereulok 11, 119002 Moscow, Russia}

\email[Alexander~Blokh]{ablokh@math.uab.edu}
\email[Peter~Ha\"issinsky]{peter.haissinsky@math.cnrs.fr}
\email[Lex~Oversteegen]{overstee@math.uab.edu}
\email[Vladlen~Timorin]{vtimorin@hse.ru}



\begin{abstract}
Holomorphic renormalization plays an important role in complex polynomial dynamics.
We consider invariant continua that are not polynomial-like Julia sets because of extra critical points.
However, under certain assumptions, these invariant continua can be identified with Julia sets of
 lower degree polynomials up to a topological conjugacy.
Thus we extend the concept of renormalization.
\end{abstract}

\maketitle
\section{Introduction}
\label{s:intro}

One-dimensional holomorphic dynamics can be viewed as a natural toy model for
various phenomena yielding rigorous results that can be transferred to other
dynamical systems. An important example here is the concept of
renormalization. It appears in many contexts but is especially closely
studied for polynomial maps for whom Douady and Hubbard \cite{DH-pl}
introduced the notion of polynomial-like mappings. Such mappings provide an
efficient framework to study renormalization. In the present paper, we
propose a new setting under which polynomials exhibit renormalization.

We start our Introduction by providing the necessary definitions. Then the
main result of this paper is stated. Finally, we explain its relevance by
illustrating it in different contexts.

Consider a degree $d>1$ polynomial $P:\C\to\C$ and a full $P$-invariant continuum $Y\subset\C$.
Say that $P:Y\to Y$ is a \emph{degree $k$ branched covering} if there is a degree $k$ branched covering $\tilde P:U\to \tilde P(U)$
 where $U$ is a neighborhood of $Y$, we have $\tilde P|_Y=P|_Y$, and $Y$ is a component of $\tilde P^{-1}(Y)$.

Points of $\ol{P^{-1}(Y)\sm Y}\cap Y$ are called \emph{irregular points} of $Y$.
A point $y\in Y$ is irregular if arbitrarily close to $y$ there are points $y'$ that do not belong to $Y$ but map into $Y$.
Since $P|_Y$ is locally onto, for each such $y'$ there is a point $y''\in Y$ close to $y'$ and such that $P(y'')=P(y')$.
It follows that $y$ is critical.
Thus, all irregular points of $Y$ are critical; the converse is not true in general.
The main result of this paper is the following.

\begin{Mthm}
Let $P:\C\to\C$ be a polynomial.
Consider a full $P$-invariant continuum $Y\subset\C$ and an integer $k>1$ such that:
\begin{enumerate}
  \item the map $P:Y\to Y$ is a degree $k$ branched covering;
  \item all irregular points of $Y$ are eventually mapped to repelling periodic points;
  \item the immediate basins of all attracting or parabolic points in $Y$ are subsets of $Y$.
\end{enumerate}
Then $P:Y\to Y$ is topologically (in fact, quasi-symmetrically) conjugate to $Q|_{K_Q}$,
 where $Q$ is a polynomial of degree $k$ with connected filled Julia set $K_Q$.
\end{Mthm}

Fig. \ref{fig:apl1} shows (in dark grey) an invariant set $Y$ for polynomial $P(z)=z(z+2)^2$.
The set $Y$ satisfies the assumptions of the Main Theorem.
In particular, $P|_Y$ is topologically conjugate to $Q(z)=-z+z^2$, even though $Y$ is not a polynomial-like (PL) filled Julia set.
A statement similar to the Main Theorem first appeared in \cite[Prop. 1, Ch. 5]{hai98-phd}.
It was made in a more general context of \emph{polynomial figures}.

\begin{figure}
  \centering
  \includegraphics[width=10cm]{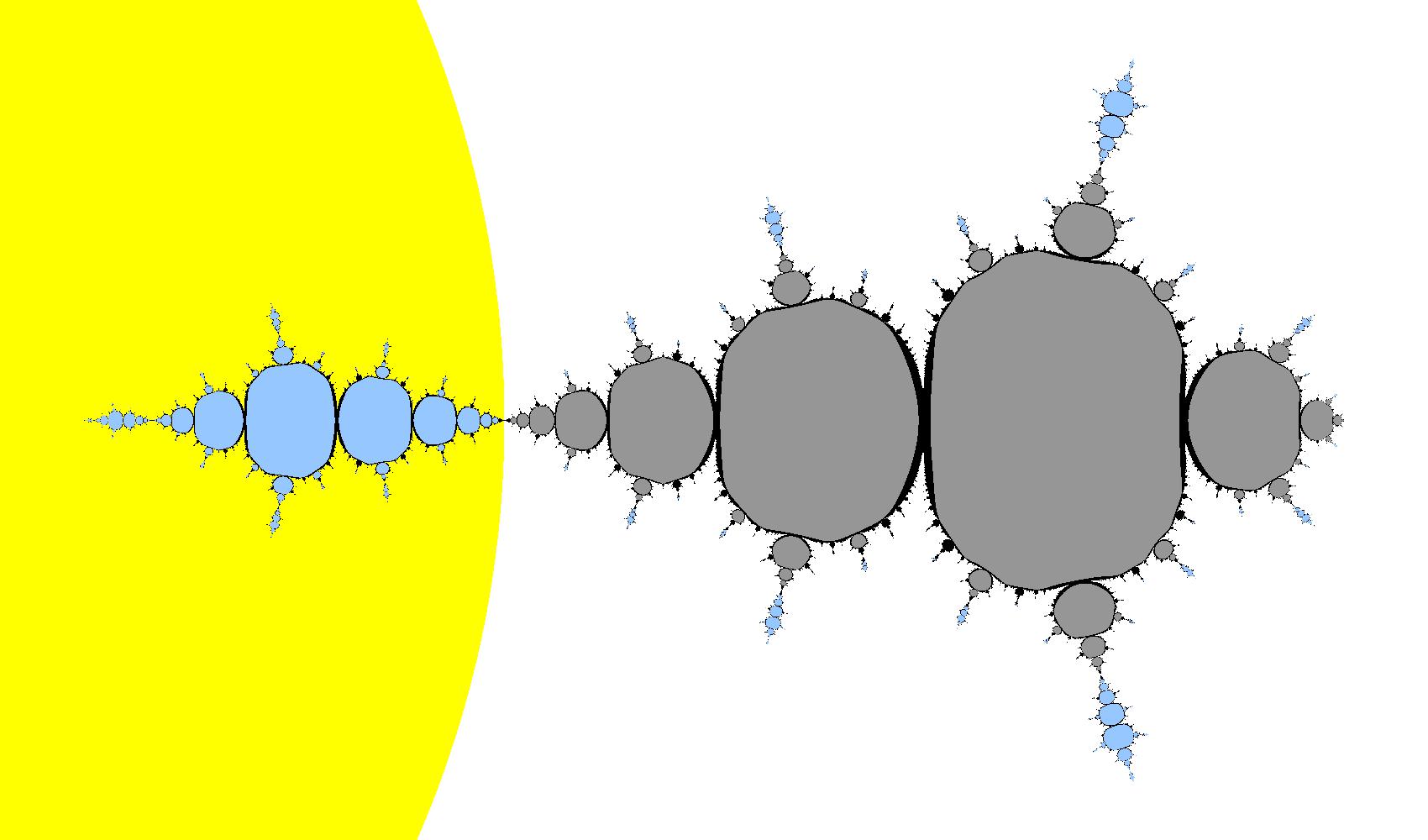}
\caption{An invariant set $Y$ for the polynomial $P(z)=z(z+2)^2$ (shown in dark grey).}
 \label{fig:apl1}
\end{figure}

The Main Theorem fits into the following paradigm:
 extrinsic properties of an invariant subset $Y\subset K_P$ imply intrinsic (structural) properties.
A pioneering work in higher-dimensional real dynamics in which the structure of hyperbolic sets under various
assumptions is described fits into that paradigm (see, e.g., \cite{hk95} and references thereof).
In the context of complex polynomial dynamics, it was advanced by the Straightening Theorem of Douady and Hubbard \cite{DH-pl}
 (see Theorem \ref{t:DH-pl} below), where polynomial-like behavior of $P$ \emph{outside} of the corresponding
 PL filled Julia set $Y$ implies a hybrid conjugacy \emph{on} $Y$.
A folklore result of Theorem \ref{t:poly-like} (= Theorem B of \cite{bopt16a})
 gives an equivalent but easier to verify extrinsic conditions on $Y$.
The Main Theorem is a partial extension of Theorem \ref{t:poly-like}, in which both the assumptions
 and conclusions are weaker.

Below, four sample applications of the Main Theorem are mentioned.

\subsection{Planar fibers}
The Main Theorem applies to the
\emph{planar fibers} (the notion is due to Schleicher \cite{sch99} and was studied
in other papers, e.g., in \cite{bclos16}).
Suppose that $K_P$ is connected.
Call a periodic repelling or parabolic point, or an iterated preimage thereof,
a \emph{valuable} point. If $z$ is a valuable point at which
more than one external ray lands, call the union $\mathrm{Cu}_z$ of $z$ with all external rays landing at
$z$ the \emph{star cut (at $z$)}.
The set $\mathrm{Cu}_z$ partitions $\C$ into finitely many open wedges.
A \emph{planar fiber (of $P$)} is a non-empty intersection
of the closures of open wedges chosen at every valuable point with a star cut.
It follows that a planar fiber is the union of a full subcontinuum of $K_P$ and various rays, and that planar fibers
map onto planar fibers. Planar fibers are important for studying symbolic dynamics of $P$
and relating the dynamics of $P$ and that of $z^d|_{\uc}$ even in the case when $K_P$ has no good topological properties.

Let $F$ be a periodic planar fiber of $P$ of minimal period $m$ and set $Y=F\cap K_P$.
It can be shown that $P^m:Y\to Y$ is a degree $k$ branched covering;
moreover, $k=1$ implies that $Y$ is a repelling periodic point.
A detailed argument for the latter claim ($Y$ is a singleton if $k=1$) is given in Theorem \ref{t:1-1} (cf. \cite{bfmot13}).
On the other hand, if $Y$ is non-degenerate, then it must contain a non-repelling periodic point of period $m$ and a critical point.
The Main Theorem implies the following corollary concerning planar fibers. By an \emph{outward parabolic} point of $Y$
we mean a parabolic point in $Y$ whose basin is not in $Y$.

\begin{cor}
  \label{c:main1}
Let $F$ be a periodic planar fiber
of a polynomial $P$ of minimal period $m$ and set $Y=F\cap K_P$.
If $Y$ has no outward parabolic points, or parabolic points that are eventual images of
irregular critical points in $Y$, then
 $P^m|_Y$ is topologically conjugate to $Q|_{K_Q}$ for some polynomial $Q$ of degree $>1$.
Moreover, $K_Q$ has a non-repelling fixed point and no valuable cutpoints.
\end{cor}

Corollary \ref{c:main1} follows from the Main Theorem since,
 as we show in Section \ref{s:fk-apl}, all irregular points of $Y$ are necessarily preperiodic
 and accessible from the basin of infinity.
It follows that they are eventually mapped to repelling or parabolic cycles;
 on the other hand, the case when they are mapped to outward parabolic points is excluded by the assumptions of Corollary \ref{c:main1}.
See Section \ref{s:fk-apl} for more details.

\subsection{Inou--Kiwi straightening domains}
More generally, invariant or periodic continua appear in the context of Inou--Kiwi renormalization \cite{IK12}.
Recall that star cuts of a polynomial $T$ give rise to equivalence classes of a certain equivalence relation $\lambda_T$ on $\Q/\Z$.
This relation $\lambda_T$ is called the \emph{rational lamination} of $T$.
For brevity say that a star cut comes \emph{from $\lambda_T$} if the arguments of the rays in the
star cut form one $\lambda_T$-class.
An interesting case is when a polynomial $P_0$ is hyperbolic and
$\lambda_P\supset\lambda_{P_0}$ for some polynomial $P$.
For a fixed $P_0$, the set of such monic centered polynomials $P$ is denoted by $\Cc(\lambda_{P_0})$.
One is tempted to think of $P$ as a result of \emph{tuning} applied to $K_{P_0}$,
 that is, an operation replacing the closures of attracting basins (and their iterated pullbacks)
 with filled Julia sets of suitable degree.
However, it is sometimes difficult to make this understanding precise.

Let $U$ be a periodic Fatou domain of $P_0$ of minimal period $m$.
Then there is the corresponding period $m$ continuum $Y_U\subset K_P$ for any $P\in \Cc(\lambda_{P_0})$.
In order to define $Y_U$, consider only the star cuts of $P$ that come from $\lambda_{P_0}$
 and the corresponding wedges; call the latter \emph{$\lambda_{P_0}$-wedges} of $P$.
There is a natural one-to-one correspondence between the $\lambda_{P_0}$-wedges of $P$ and the wedges of $P_0$.
By definition, the set $Y_U$ is the intersection of the closures of all $\lambda_{P_0}$-wedges of $P$
 such that the corresponding wedges of $P_0$ contain $U$.
Note that $Y_U$ is not always a polynomial-like Julia set, since it may contain ``unwanted'' critical or parabolic points.
Given $P_0$, the set of all $P\in\Cc(\lambda_{P_0})$ such that, for all $U$ as above,
 $Y_U$ are polynomial-like, is denoted by $\Rc(\lambda_{P_0})$.
This is the domain of the \emph{straightening map}.
It is proved in \cite{IK12} that $\Rc(\lambda_{P_0})=\Cc(\lambda_{P_0})$ if and only if $P_0$ is \emph{primitive},
 i.e., the closures of distinct bounded Fatou components of $P_0$ are disjoint.
Straightening maps of Inou--Kiwi type have been studied in several recent papers, see e.g. \cite{I18,sw20,w21}.
The Main Theorem allows one to extend the straightening maps to certain elements of
 $\Cc(\lambda_{P_0})\setminus\Rc(\lambda_{P_0})$.

\begin{cor}
  \label{c:Y_U}
 Let $Y_U\subset K_P$ be as above.
 If $Y_U$ has no outward parabolic points, or parabolic points that are eventual images of
irregular critical points in $Y$, then $P^m|_{Y_U}$ is topologically conjugate to $Q|_{K_Q}$
 for some polynomial $Q$ whose degree coincides with that of $P_0^m:U\to U$.
\end{cor}

The proof of Corollary \ref{c:Y_U} is similar to that of Corollary \ref{c:main1}, see Section \ref{s:fk-apl}.
We believe that our techniques extend to arbitrary mapping schemata,
 with essentially the same arguments, but leave necessary modifications to the reader.

\subsection{A special case of the Douady conjecture}
An irrational number $\theta\in\R\sm\Q$ is said to be \emph{Brjuno} if $B(\theta)<\infty$, where
$$
B(\theta)=\sum_{n=1}^{\infty}\frac{\log q_{n+1}}{q_n}
$$
is the \emph{Brjuno function}, and $p_n/q_n\to \theta$ are the continued fraction convergents for $\theta$.
By a theorem of Brjuno \cite{brj71}, $B(\theta)<\infty$ implies that any holomorphic germ of the form $f(z)=e^{2\pi i\theta} z+\dots$
 is linearizable.
Yoccoz \cite{yoc} proved a partial converse: if a quadratic polynomial $Q(z)=e^{2\pi i\theta} z+z^2$ is linearizable,
 then $\theta$ is Brjuno.
The \emph{Douady conjecture} states that this is also true for higher-degree polynomials
(see \cite{dou87}).
A new proof of a special case of the Douady conjecture can be deduced from the Main Theorem.

\begin{cor}
  \label{c:dou}
  Let $P(z)=e^{2\pi i\theta}z+\dots$, where $\theta\in\R\sm\Q$, be a cubic polynomial with at least one (pre)periodic critical point.
Then $P$ is linearizable at $0$ if and only if $\theta$ is Brjuno.
\end{cor}

The proof of Corollary \ref{c:dou} is given in Section \ref{s:fk-apl}.
Note that the same result also follows from \cite{bc11}, which, however, uses essentially different methods.
Recall \cite[Corollary 7.7]{bcot21} that the conclusion of Corollary \ref{c:dou} holds
 whenever a cubic polynomial $P(z)=e^{2\pi i\theta}z+\dots$
 is not in the closure $\Pc$ of the principal hyperbolic component $\Pc^\circ$ (the one containing $z^3$).
Namely, if $P\notin\Pc$, and $\theta$ is not Brjuno, then $P$ is not linearizable at $0$, that is, $0$ is a Cremer point.
Corollary \ref{c:dou} yields the same conclusion even for some $P\in\Pc$.
In fact, if a strictly preperiodic critical point of $P$ belongs to the planar fiber of $0$, then $P\in\Pc$.
This follows from \cite[Corollary D]{bcot21} and the fact that a critical point of $P$ being preperiodic
 implies that $P$ cannot lie inside a stable domain of the slice of cubic polynomials $f$ given by the conditions
 $f(0)=0$, $f'(0)=e^{2\pi i\theta}$.

\subsection{Multipliers of periodic points}
The next Corollary follows immediately from the Main Theorem and,
in the irrational neutral  case, results by Perez-Marco \cite{P-M}.

\begin{cor}
\label{c:pm}
Assuming the conditions of  the Main Theorem, let $\vp$ be a topological conjugacy between $P|_Y$
and $Q|_{K_Q}$, and let $C$ be a periodic cycle in $Y$.
Then $C$ is attracting (resp., repelling, neutral) if and only if $\vp(C)$ is attracting (resp., repelling, neutral).
Moreover, if $C$ is non-repelling, then it has the same multiplier as $\vp(C)$.
\end{cor}

Observe also that (as follows from the proof of the Main Theorem) $\vp$ extends to a
 quasi-conformal embedding of an open set containing $Y$.
Therefore, the boundary of $Y$ has positive area if and only if the corresponding quadratic Julia set $J_Q$ does.
Recall that Buff and Ch\'eritat \cite{bc12} gave examples of quadratic polynomial Julia sets of positive area.
These translate into examples of non-renormalizable cubic polynomials whose Julia sets have positive measure.

\section{Preliminaries}
\label{s:prelim}

Throughout, let $P:\C\to\C$ be a polynomial of degree $d>1$ with connected filled Julia set $K_P$.
Clearly, $P$ acts on the Riemann sphere $\ol\C=\CP^1$ so that $P(\infty)=\infty$.
In contrast to rational dynamics, the point at infinity plays a special role in the dynamics of $P$.
A classical theorem of B\"ottcher states that $P$ is conjugate to $z\mapsto z^d$ near infinity.
Since $K_P$ is connected, the conjugacy can be defined on $\ol\C\sm K_P$ as follows.
We will write $\disk=\{z\in\C\,|\, |z|<1\}$ for the open unit disk in $\C$ and $\ol\disk$ for its closure.
Without loss of generality we may assume that $P$ is \emph{monic}, i.e., the highest order term of $P$ is $z^d$.
Let $\psi_P:\disk\to\ol\C\sm K_P$ be a conformal isomorphism normalized so that $\psi_P(0)=\infty$ and $\psi'_P(0)>0$.
Then $\psi_P^{-1}\circ P\circ\psi_P$ is a degree $d$ holomorphic self-covering of $\disk$.
The only option for such a holomorphic self-covering is $z\mapsto \la z^d$ with $|\la|=1$.
By the chosen normalization of $P$ and $\psi_P$, the coefficient $\la$ must be equal to $1$.
It follows that $P(\psi_P(z))=\psi_P(z^d)$ for any $z\in \C$.

Thus, if we use the polar coordinates $(\ta,\rho)$ on $\disk$ and identify $\disk$ with $\ol\C\sm K_P$ by $\psi_P$,
 then the action of $P$ will look like $(\ta,\rho)\mapsto (d\ta,\rho^d)$.
Here $\ta$ is the angular coordinate; it takes values in $\R/\Z$ (elements of $\R/\Z$ are called \emph{angles}).
The coordinate $\rho$, the radial coordinate, is the distance to the origin.
On $\disk\sm\{0\}$ (hence, after the transfer, on $\C\sm K_P$), it takes values in $(0,1)$.
\emph{External rays} of $P$ are defined as the $\psi_P$-images of radial straight intervals in $\disk$.

\subsection{Rays and equipotentials}
\label{ss:rays}
Consider a straight radial interval $\r(\ta)=\{e^{2\pi i\ta}\rho\,|\,\rho\in (0,1)\}$ from $0$ to the point $e^{2\pi i\ta}$.
The \emph{external ray of $P$ of argument} $\ta\in\R/\Z$ is the set $R_P(\ta)=\psi_P(\r(\ta))$.
External rays are useful for studying the dynamics of $P$.
In particular, it is important to know when different rays land at the same point.

\begin{dfn}[Ray landing]
\label{d:rayland}
 A ray $R_P(\ta)$ \emph{lands} at $a\in K_P$ if $a=\lim_{\rho\to 1^-} \psi_P(e^{2\pi i\ta}\rho)$
is the only accumulation point of $R_P(a)$ in $\C$.

By the Douady--Hubbard--Sullivan landing theorem, if $\ta$ is rational, then $R_P(\ta)$ lands at a (pre)periodic point
that is eventually mapped to a repelling or parabolic periodic point.
Conversely, any point that eventually maps to a repelling or parabolic periodic point is the
 landing point of at least one and
at most finitely many external rays with rational arguments.
\end{dfn}

An \emph{equipotential curve} of $P$ (or simply an \emph{equipotential}) is the $\psi_P$-image of a circle $\{z\in\C\,|\, |z|=\rho\}$
 of radius $\rho\in (0,1)$ centered at $0$.
External rays and equipotentials form a net that is the $\psi_P$-image of the polar coordinate net.

\subsection{Quasi-regular and quasi-symmetric maps}
Let us recall the definition of quasi-regular \cite{ric} and quasi-conformal maps \cite{ahl66}.

\begin{dfn}[Quasi-regular maps]
\label{d:qr}
Let $U$ and $V$ be open subsets of $\C$, and let $\kappa\ge 1$ be a real number.
A map $f:U\to V$ is said to be \emph{$\kappa$-quasi-regular} if it has distributional partial derivatives in $L^2_{loc}$, and
 $||df||^2\le \kappa\, \jac_f$ in $L^1_{loc}$.
Here $df$ is the first differential of $f$, and $\jac_f$ is the Jacobian determinant of $f$.
Note that any holomorphic map is $\kappa$-quasi-regular with $\kappa=1$.
We say that $f$ is \emph{quasi-regular} if it is $\kappa$-quasi-regular for some $\kappa\ge 1$.
A \emph{quasi-conformal map} is by definition a quasi-regular homeomorphism.
All these maps are orientation-preserving by definition ($\jac_f\ge 0$ follows from the inequality displayed above).
\end{dfn}

The inverse of a ($\kappa$-)quasi-conformal map is ($\kappa$-)quasi-conformal.
Quasi-conformal maps admit a number of analytic and geometric characterizations.
They can be characterized in terms of Beltrami differentials and in terms of moduli of annuli or similar conformal invariants.
See 4.1.1 and 4.5.16 --- 4.5.18 in \cite{hub}.
A metric characterization of quasi-conformal maps is based on the following notion applicable to general metric spaces,
cf. \cite{tv80}.

\begin{dfn}[Quasi-symmetric maps]
\label{d:qs}
Let $(X,d_X)$ and $(Y,d_Y)$ be metric spaces, and let $\eta:[0,\infty)\to [0,\infty)$ be an increasing onto homeomorphism.
A continuous embedding $f:X\to Y$ is said to be \emph{quasi-symmetric of modulus $\eta$} (or \emph{$\eta$-quasi-symmetric}) if
\begin{equation}\label{eq:qsmap}
  \frac{d_Y(f(x),f(y))}{d_Y(f(x),f(z))}\le \eta\left(\frac{d_X(x,y)}{d_X(x,z)}\right)
\end{equation}
for all $x\ne y\ne z$ that are sufficiently close to each other.
We will sometimes abbreviate quasi-symmetric as QS.
The inverse of a QS embedding (defined on $f(X)$) is $\eta'$-QS, where $\eta'(t) = 1/\eta^{-1}(1/t)$.
The composition of QS embeddings is also QS.
A continuous embedding $f:X\to Y$ is \emph{$\kappa$-weakly QS} for some $\kappa>0$ if
$$
d_X(x,y)\le d_X(x,z)\quad\Longrightarrow\quad d_Y(f(x),f(y))\le\kappa\, d_Y(f(x),f(z)).
$$
\emph{Weakly QS embeddings} are $\kappa$-weakly QS for some $\kappa>0$.
Clearly, QS embeddings are weakly QS.
The converse is not true in general, however, by Theorem 10.19 of \cite{hei}, weakly QS embeddings are QS in a lot of cases.
In particular, a weakly QS embedding of a connected subset of $\R^n$ to $\R^n$ is QS.
Occasionally we will talk about ``QS maps'' which will always mean ``QS embeddings''.
\end{dfn}

The following theorem establishes a relationship between QS embeddings and quasi-conformal maps.

\begin{thm}[A special case of Theorems 2.3 and 2.4 of
\cite{vai81}]
  \label{t:qs-qc}
  An $\eta$-QS embedding between domains in $\C$ is $\kappa$-quasi-conformal
($\kappa\ge 1$ is a constant depending only on $\eta$).
Conversely, consider a $\kappa$-quasi-conformal map $f:U\to V$, where $U$, $V\subset\R^2$ are open.
Then, for any $z\in U$ and $\eps>0$ such that the $2\eps$-neighborhood of $z$ lies in $U$,
 the map $f$ is $\eta$-QS on the $\eps$-neighborhood of $z$, where $\eta$ depends only on $\kappa$.
\end{thm}

Quasi-conformal images of circle arcs, circles, and disks can be described explicitly.

\begin{dfn}[Quasi-arc, quasi-circle, quasi-disk]
A \emph{simple arc} in $\C$ is the image of $[0, 1]$ under a homeomorphic embedding $\xi:[0, 1]\to \C$.
A simple arc $I$ is a \emph{quasi-arc} if for any such $\xi$ and any
$x\le y\le z$ we have
\begin{equation}\label{eq:qa}
  |\xi(x)-\xi(z)|\ge C|\xi(x)-\xi(y)|,
\end{equation}
where $C>0$ is a constant independent of $x$, $y$, $z$ and $\xi$.
A \emph{quasi-circle} is a Jordan curve such that any sufficiently small arc of it is a quasi-arc with a uniform constant $C$.
For quasi-arcs and quasi-circles in the \emph{Riemann sphere}
$\ol\C=\mathbb{C}P^1$, we use the spherical distance between $a$ and
$b$ instead of $|a-b|$. %, etc.
A \emph{quasi-disk} is a Jordan disk
bounded by a quasi-circle. A \emph{quasi-reflection} in a
 Jordan curve is an orientation-reversing involution
 of the sphere that (1) switches the inside and the outside of the curve fixing points on the curve,
 (2) upon post-composition with any anti-holomorphic homeomorphism of the sphere, produces a quasi-conformal map.
\end{dfn}

The following theorem is due to L. Ahlfors, see \cite{ahl66} or 4.9.8, 4.9.12, and 4.9.15 in \cite{hub}:

\begin{thm}
  \label{t:ahl}
  Properties $(1) - (3)$ of a Jordan curve $S$ are equivalent:
  \begin{enumerate}
    \item the curve $S$ is a quasi-circle;
    \item there is a bi-Lipschitz quasi-reflection in $S$;
    \item there is a quasi-conformal map $h:\ol\C\to\ol\C$ such that $S=h(\ol\R)$.
  \end{enumerate}
\end{thm}

Observe also that QS embeddings of quasi-arcs are quasi-arcs;
 moreover, preimages of quasi-arcs under QS-embeddings are quasi-arcs too.

Quasi-symmetric maps between quasi-circles can be extended inside the corresponding quasi-disks as
quasi-conformal maps.

\begin{thm}
  \label{t:qs-ext}
If $U$ and $V$ are quasi-disks in $\ol\C$, and $f:\bd(U)\to\bd(V)$ is a quasi-symmetric map, then
there is a continuous map $F:\ol U\to\ol V$ such that $F=f$ on $\bd(U)$, and $F$ is quasi-conformal in $U$.
\end{thm}

\begin{proof}
By Theorem \ref{t:ahl}, there are quasi-conformal maps $h_U$, $h_V:\ol\C\to\ol\C$ that take the upper half-plane
$\H=\{z\in\C\,|\,\Im z>0\}$ onto $U$, $V$, respectively.
Then the map $\vp=h_V^{-1}\circ f\circ h_U:\ol\R\to\ol\R$ is quasi-symmetric as a composition of quasi-symmetric maps.
Pre-composing $h_U$ and $h_V$ with suitable real fractional linear maps, arrange that $\vp(\infty)=\infty$.
Let $\eta$ be a modulus of $\vp$ (so that $\vp$ is $\eta$-quasi-symmetric).
Setting $y=x+t$ and $z=x-t$ in the definition of an $\eta$-quasi-symmetric map, we see that
$$
M^{-1}\le\frac{\vp(x+t)-\vp(x)}{\vp(x)-\vp(x-t)}\le M,
$$
where $M=\eta(1)$.
Maps $\vp$ that satisfy the above condition for some $M>0$ are called \emph{$\R$-quasi-symmetric} in \cite{hub}.
The constant $M$ is called the \emph{modulus} of an $\R$-quasi-symmetric map.
By a theorem of Ahlfors and Beurling \cite{ab56} (see also \cite{ahl66} and 4.9.3 and 4.9.5 of \cite{hub}),
 an $\R$-quasi-symmetric map of modulus $M$ admits a $\kappa$-quasi-conformal extension in $\H$, where $\kappa$ depends only on $M$.
More precisely, there is a continuous map $\Phi:\ol\H\to\ol\H$ such that $\Phi=\vp$ on $\ol\R$,
 and $\Phi|_{\H}$ is $\kappa$-quasi-conformal.
Then $F=h_V\circ\Phi\circ h_U^{-1}$ has the desired property.
\end{proof}

\subsection{Straightening}
Let $U$ and $V$ be Jordan disks such that $U\Subset V,$ that is, $\ol U$ is a compact subset of $V$.
Recall the following classical definitions of Douady and Hubbard \cite{DH-pl}.

\begin{dfn}[Polynomial-like maps]
\label{d:pl}
  Let $f:U\to V$ be a proper holomorphic map.
Then $f$ is said to be \emph{polynomial-like} (PL).
The \emph{filled Julia set} $K(f)$ of $f$ is defined as the set of points in $U$, whose forward $f$-orbits stay in $U$.
\end{dfn}

Similarly to polynomials, the set $K(f)$ is connected if and only if all critical points of $f$ are in $K(f)$.

\begin{dfn}[Hybrid equivalence]
\label{d:hybeq}
 Let $f_1:U_1\to V_1$ and $f_2:U_2\to V_2$ be two PL maps.
Consider Jordan neighborhoods $W_1$ of $K(f_1)$ and $W_2$ of $K(f_2)$.
A quasiconformal homeomorphism $\phi:W_1\to W_2$ is called a \emph{hybrid equivalence} between $f_1$ and $f_2$
 if $f_2\circ\phi=\phi\circ f_1$ whenever both parts are defined, and  $\ol\d\phi=0$ on $K(f_1)$.
\end{dfn}

Recall the following classical theorem of Douady and Hubbard \cite{DH-pl}.

\begin{thm}[PL Straightening Theorem]
  \label{t:DH-pl}
  A poly\-no\-mi\-al-like map $f :U\to V$ is hybrid equivalent to a polynomial of the same degree restricted on a
Jordan neighborhood of its filled Julia set.
\end{thm}

Theorem \ref{t:poly-like} below appears to be a folklore result.
It is formally proved, e.g., in \cite{bopt16a} (Theorem B).

\begin{thm}
  \label{t:poly-like}
  Let $P:\C\to\C$ be a polynomial, and $Y\subset\C$ be a non-se\-pa\-ra\-ting
  $P$-invariant continuum.
  The following assertions are equivalent:
\begin{enumerate}
 \item the set $Y$ is the filled Julia set of some polynomial-like
     map $P:U\to V$ of degree $k$,
 \item the set $Y$ is a component of the set $P^{-1}(P(Y))$ and, for every
     attracting or parabolic point $y$ of $P$ in $Y$, the
     attracting basin of $y$ or the union of all parabolic domains
     at $y$ is a subset of $Y$.
\end{enumerate}
\end{thm}

We will need quasi-regular maps whose topological properties resemble those of polynomials.

\begin{dfn}
  \label{d:qr-poly}
A quasiregular map $f:\ol\C\to\ol\C$ is called a \emph{quasiregular polynomial} if $f^{-1}(\infty)=\{\infty\}$, and
 $f$ is holomorphic near infinity.
\end{dfn}

Let us state a partial case of \cite[Theorem 5]{sw20}.

\begin{thm}
\label{t:sw}
  Let $f:\C\to\C$ be a quasi-regular polynomial of degree $d\ge 2$ and let $A\subset\C$ be a Borel set
such that $\ol\d f=0$ a.e. outside $A$.
Assume that there is a positive integer $T$ such that, for every $z$,
 the set of nonnegative integers $k$ with $f^k(z)\in A$ has cardinality $\le T$.
Then there is a QC map $\Psi:\C\to\C$ and a polynomial $F:\C\to\C$ of degree $d$ such that $F\circ\Psi=\Psi\circ f$.
Moreover, $\ol\d\Psi=0$ holds a.e. on the set $\{z\in\C\,|\, f^n(z)\notin A\ \forall n\ge 0\}$.
\end{thm}

This version is a rather straightforward extension of the Douady--Hubbard straightening theorem (Theorem \ref{t:DH-pl})
 and is similar to Shi\-shi\-ku\-ra's Fundamental Lemma for qc-surgery (cf. \cite[Lemma 3.1]{Shi87}),
however, the general version of \cite[Theorem 5]{sw20} is much more powerful.

\section{Avoiding sets}
\label{s:avoid}
In this section, we give more detailed statements of the main results, and outline the plan for the rest of the paper.
Throughout this section, $P:\C\to\C$ is a complex degree $d>1$ polynomial with \emph{connected} $K_P$.

\subsection{Cuts and avoiding sets}
\label{ss:genren}
If external rays $R$ and $L$ land at the same point $a$,
the union $\Ga=R\cup L\cup\{a\}$ is called a \emph{cut}.
The point $a$ is called the \emph{root point} of $\Ga$.
The cut $\Ga$ is \emph{degenerate} if $R=L$ and \emph{nondegenerate} otherwise.
A subarc of a degenerate cut that contains its landing point is called a
\emph{terminal segment} of the cut. Nondegenerate cuts separate $K_P$.
A \emph{wedge} is a complementary component of a cut in $\C$; the root point of a wedge
is the root point of the corresponding cut.
We assume that cuts are oriented from $R$ to $L$ so that every cut $\Ga$ bounds a unique wedge $W=W_\Ga$
 where $\Ga$ is the \emph{oriented} boundary of $W$.
If $\Ga$ is degenerate, then we set $W_\Ga=\0$.
For a finite collection of cuts $\Zc$, let $\Wc_\Zc$ be the collection $\{W_\Ga\mid \Ga\in\Zc\}$ of the corresponding wedges
 and let $\bigcup\Wc_{\Zc}=\bigcup_{\Ga\in\Zc} W_\Ga$ denote the union of these wedges.
Say that $\Zc$ is $P$-\emph{invariant} if $P(\Ga)\in\Zc$ for every $\Ga\in\Zc$.

\begin{dfn}\label{d:avoid}
A finite set $\Zc$ of cuts is \emph{admissible} if it is $P$-invariant,
two distinct cuts from $\Zc$ can share at most a common root point, and
 $\C\sm \bigcup\Wc_{\Zc}$ is a connected set containing all $\Ga\in\Zc$;
 the latter set is called the \emph{principal set} of $\Zc$.
Let $A_P(\Zc)$ be the set of all $x\in K_P$ such that $f^n(x)$ is in the principal set, for all $n\ge 0$.
Equivalently, $x\in K_P$ belongs to $A_P(\Zc)$ if $f^n(x)\notin \bigcup\Wc_{\Zc}$ for $n\ge 0$.
The set $A_P(\Zc)$ is called the \emph{avoiding set} of $\Zc$.
\end{dfn}

Observe that if $\Zc$ is admissible then all associated wedges are pairwise disjoint.
By definition $\Ga\cap K_P\subset A_P(\Zc)$ for every $\Ga\in\Zc$.
Formally, the definition of $A_P(\Zc)$ is applicable to the case $\Zc=\0$.
If $\Zc$ is empty or consists of only degenerate cuts, then $A_P(\Zc)=K_P$.
Otherwise, $A_P(\Zc)$ is a proper subset of $K_P$.
A root point $a$ of a cut $\Ga\in \Zc$ is called \emph{outward parabolic} if $a$ is a parabolic periodic point,
 and there is a Fatou component in $K_P\sm A_P(\Zc)$ containing an attracting petal of $a$.
If a periodic root point $a$ of a cut $\Ga\in \Zc$ is not outward parabolic,
 then it is said to be \emph{outward repelling}.
Observe that an outward repelling periodic root point $a$ of a cut $\Ga\in \Zc$ may be parabolic;
 in that case Fatou components containing attracting petals of $a$ are all contained in $A_P(\Zc)$.
Define $\rt_\Zc$ as the collection of root points of all cuts from $\Zc$.
Classical arguments yield Theorem \ref{t:pl}.
Recall: we write $U\Subset V$ if $\ol{U}\subset V$ is compact.

\begin{thm}
  \label{t:pl}
Let $\Zc$ be admissible, and suppose that $A_P(\Zc)$ is connected.
If there are no critical or outward parabolic points in $\rt_\Zc$,
 then there exist Jordan domains $U\Subset V$ such that $P:U\to V$ is polynomial-like, and
 $A_P(\Zc)$ is the filled Julia set of this polynomial-like map.
In particular, $P|_{U}$ is hybrid equivalent to a polynomial $Q$ restricted to a neighborhood of $K_Q$.
If $\Zc\ne\0$, then $\deg(Q)<d$.
\end{thm}

Observe that both Definition \ref{d:pl} and Theorem \ref{t:pl} allow for the degree one case; in that case
$A_P(\Zc)$ is a repelling fixed point.

The following theorem is a constructive version of the Main Theorem.
It generalizes Theorem \ref{t:pl} to certain cases, where one cannot hope for a polynomial-like behavior.

\begin{thm}
  \label{t:main-a}
Consider an admissible collection of cuts $\Zc$.
Suppose that $A_P(\Zc)$ is connected,  every critical point of $\rt_\Zc$ is eventually mapped to a repelling periodic orbit,
 and no point of $\rt_\Zc$ is outward parabolic.
Then either $A_P(\Zc)$ is a singleton, or $P|_{A_P(\Zc)}$ is quasi-symmetrically conjugate to $Q|_{K_Q}$,
 where $Q$ is a polynomial of degree greater than one.
If $\Zc\ne\0$, then the degree of $Q$ is less than $d$.
Moreover, the conjugacy can be arranged to preserve the complex structure almost everywhere on $A_P(\Zc)$.
\end{thm}

The general case of the Main Theorem can (and will) be reduced to Theorem \ref{t:main-a}.
If $P|_{A_P(\Zc)}$ as in Theorem \ref{t:main-a} is injective, then $A_P(\Zc)$ is a singleton by Theorem \ref{t:1-1}.
Note that a quasi-symmetric conjugacy is in particular a topological conjugacy.

Fig. \ref{fig:apl1} illustrates $\Wc_\Zc$ and $A_P(\Zc)$ for a specific cubic polynomial $P$.
Namely, $P(z)=z(z+2)^2$, and the set $A_P(\Zc)$ for $P(z)=z(z+2)^2$ is shown in dark grey.
Here $\Wc_\Zc$ consists of a single wedge $W$ (highlighted on the left) whose boundary is mapped to $R_P(0)$.
The boundary rays of $W$ are $R_P(1/3)$ and $R_P(2/3)$, and the root point $a=-2$ of $W$ maps to the fixed point $0$.
By Theorem \ref{t:main-a}, the filled Julia set $K_P$ consists of a copy of $K_Q$, where $Q(z)=-z+z^2$,
 and countably many decorations.
The parabolic point $0$ of $Q$ corresponds to the parabolic point $-1$ of $P$ of the same multiplier.
The maps $P|_{A_P(\Zc)}$ and $Q|_{K_Q}$ are topologically conjugate, but $A_P(\Zc)$ is not a PL filled Julia set.

\subsection{Analogs and extensions}
Branner and Douady \cite{BD88} consider the space $\Fc^+$ of cubic polynomials $P_a(z)=z(z-a)^2$
 (in a different coordinate) such that $R_{P_a}(0)=0$.
They suggested a surgery that relates cubic polynomials from $\Fc^+$ to quadratic polynomials from
 the $1/2$-limb of the Mandelbrot set $\Mc_2$.
There is a connection with our Main Theorem in the special case considered in \cite{BD88}.
Given $P\in\Fc^+$ and $\Zc=\{R_{P}(1/3)\cup R_P(2/3)\cup\{a\}\}$, Theorem \ref{t:main-a} produces
 a quadratic polynomial $Q$ such that $P|_{A_P(\Zc)}$ is topologically conjugate to $Q|_{K_Q}$.
In \cite{BD88}, a different but closely related quadratic polynomial $Q'$ is produced.
Namely, $Q'$ is in the $1/2$-limb of $\Mc_2$ and is such that
 the first return map to the side of $\ol{R_{Q'}(1/3)\cup R_{Q'}(2/3)}$
 containing $R_{Q'}(0)$ is a result of a specific surgery applied to $P$ and $\Zc$.
Our $Q$ is then a renormalization of $Q'$.
Methods employed in the proof of the Main Theorem generalize those of Branner and Douady.
A recent extension in a different direction is given in \cite{DLS20}.

If $\rt_\Zc$ is allowed to include outward parabolic points, then the situation appears to be more involved.
One cannot hope to extend the Main Theorem in its present form,
 as the QS geometry of the pieces of $A_P(\Zc)$ ``squeezed'' in cusps of parabolic domains
 is different from the QS geometry of a polynomial filled Julia set near a repelling periodic point.
On the other hand, a topological rather than QS conjugacy may still exist.
David maps can be used for a transition between parabolic and repelling periodic points (see \cite{hai98,BF14}).
Lomonaco in \cite{Lom15} introduced the theory of parabolic-like maps.
The corresponding straightening theorem is applicable to an admissible collection $\Zc=\{\Ga\}$
 of just one cut $\Ga$ with a fixed parabolic root point;
 it replaces the complement of $A_P(\Zc)$ with a single parabolic domain.
However, the latter surgery does not change parabolic dynamics to repelling one.

\subsection{Plan of the paper}
Section \ref{s:apl}
discusses the notion of transversality and its relationship with quasi-symmetric maps.
In Section \ref{s:rep} we reduce
 Theorem \ref{t:main-a} to the case when $\Zc$ has specific properties (e.g.,
one may assume that all periodic cuts in $\Zc$ are degenerate); this is done with the help of classical theory of polynomial-like maps of Douady and Hubbard.
The proof of the Main Theorem is given in Section \ref{s:plan}, where $P$ is replaced with a quasi-regular map $f$
 such that $P=f$ on $K_P$, and $f$ repels points off $K_P$.
Straightening the map $f$ using Theorem \ref{t:sw} yields Theorem \ref{t:main-a}.
Section \ref{s:fk-apl} deduces the Main Theorem and Corollaries \ref{c:main1} and \ref{c:Y_U} from Theorem \ref{t:main-a}.

\section{Transversality}
\label{s:apl}
Consider two simple arcs $R$, $L\subset\C$ sharing an endpoint $a$ and otherwise disjoint.
The arcs $R$, $L$ are \emph{transverse} at $a$ if, for any sequences $u_n\in R$ and $v_n\in L$ converging to $a$,
$$
\frac{u_n-a}{v_n-a}\not\to 1.
$$
Transversality is related to the notion of a quasi-arc as the following lemma explicates.

\begin{lem}
  \label{l:qa-trans}
Let simple arcs $R$, $L$ share an endpoint $a$ and be otherwise disjoint. If $R\cup L=I$ is a
quasi-arc, then $R$ and $L$ are transverse.
\end{lem}

\begin{proof}
By way of contradiction, suppose that
$$
\frac{u_n-a}{v_n-a}\to 1.
$$
for some $u_n\in R$ and $v_n\in L$ such that $u_n$, $v_n\to a$.
It follows that
$$
\frac{u_n-v_n}{v_n-a}\to 0.
$$
This contradicts the inequality $|v_n-u_n|\ge C|v_n-a|$ with $C>0$ from the definition of a quasi-arc.
\end{proof}

\begin{prop}
\label{p:invt}
Consider a simple arc $R$ such that the image $R'$ of $R$ under $w\mapsto w^k$ with $k>1$ is a simple arc,
$0$ is an endpoint of $R$, and $\lambda R'\supset R'$ for some $\lambda\in\C$ with $|\lambda|>1$.
Then $R$ is transverse to $\zeta R$ at $0$, for every $k$-th root of unity $\zeta\ne 1$.
\end{prop}

\begin{proof}
The map $w\mapsto w^k$ is injective on the arc $R$.
Indeed, by our assumption, $R'$ is a simple arc; a locally injective continuous map from an interval to an interval is injective.
Thus, $R$ and $\zeta R$ share only the endpoint $0$.
By way of contradiction, assume that $v_n/u_n\to \zeta$, where $u_n$, $v_n\in R$ and $u_n$, $v_n\to 0$.
Passing to a subsequence and choosing positive integers $m_n$ properly, we may assume that $\lambda^{m_n}u^k_n\to u\ne 0$, where $u\in R'$ is not an endpoint of
$R'$.
Then also $\lambda^{m_n}v^k_n\to u$.
Let $I_n$ be the segment of $R$ connecting $u_n$ and $v_n$.
Then the corresponding segment $I'_n$ of $R'$ connects $u^k_n$ with $v^k_n$.
Consider the arc $T_n=\lambda^{m_n}I'_n\subset R'$; its endpoints $\lambda^{m_n}u^k_n$ and $\lambda^{m_n}v^k_n$
converge to $u$ but the arc itself has diameter bounded away from  $0$ as it
makes one or several loops around $0$
(if $S_n$ is the union of $T_n$ and the straight segment connecting its endpoints
 then, since  $v_n/u_n\to \zeta$, the loop $S_n$ has a nonzero winding number with respect to $0$).
Since sets $T_n$ are subarcs of $R'$, in the  limit they converge to a nondegenerate loop in $R'$,
a contradiction.
\end{proof}

The following is a typical application of Proposition \ref{p:invt}.
Let $P$ be a polynomial; consider a repelling fixed point $a$ of $P$ and an invariant ray $R_P(\ta')$ landing at $a$.
Set $\lambda=P'(a)$, that is, $\lambda$ is the multiplier of the fixed point $a$.
Then, in some local coordinate $y$ near $a$, we have $y=0$ at $a$, and $P$ coincides with $y\mapsto\lambda y$.
On the other hand, suppose that a ray $R_P(\ta)$ maps to $R_P(\ta')$ under $P$.
Write $b$ for the landing point of $R_P(\ta)$, and assume that $P$ has local degree $k>1$ at $b$
(thus, $b$ is critical).
Then, in some local coordinate $x$ near $b$ combined with the local coordinate $y$ near $a$,
the map $P$ looks like $y=x^k$, and $b$ is the point where $x=0$.
We can now define $R$ as an arc of $\ol R_P(\ta)$ connecting $b$ with some point of $R_P(\ta)$.
Set $R'=P(R)$.
Then Proposition \ref{p:invt} is applicable to arcs $R$, $R'$ and the chosen local coordinates.
It claims that $R$ is transverse to all other $P$-pullbacks of $R'$ originating at $b$.

\begin{lem}
\label{l:RL-qs}
Consider simple arcs $R,$ $L$ such that their respective images $R'$, $L'$ under $x\mapsto x^k$ with $k>1$ are simple arcs,
$0$ is an endpoint of $R$ and of $L$, and $\lambda R'\supset R', \lambda L'\supset L'$ for some $\lambda\in\C$ with $|\lambda|>1$.
If $R'\cap L'=\{0\}$ and $R'$, $L'$ are transverse, then the restriction of $x\mapsto x^k$ to $R\cup L$ is QS.
\end{lem}

If $R'\cup L'$ is a quasi-arc, then, by Lemma \ref{l:RL-qs}, the arc $R\cup L$ is also a quasi-arc.
This observation will be useful in what follows.

\begin{proof}
We will prove that $x\mapsto x^k$ restricted to $R\cup L$ is weakly QS.
Assume, by way of contradiction, that there are three sequences $u_n$, $v_n$, $w_n\in R\cup L$ such that
$$
\Delta_n=\frac{u^k_n-v^k_n}{u^k_n-w^k_n}\to\infty,\quad |u_n-v_n|\le |u_n-w_n|.
$$
Assume that $u_n\to u$, $v_n\to v$ and $w_n\to w$ by passing to subsequences.
If $u\ne w$, then $\Delta_n$ would be bounded since $x\mapsto x^k$ is injective on $R\cup L$, a contradiction.
Thus $u=w$. Since  $|u_n-v_n|\le |u_n-w_n|$, it follows that  $u=v=w$.
As $x\mapsto x^k$ is locally QS away from the origin by the Koebe distortion theorem (being locally univalent), the only possibility for
$\Delta_n\to \infty$  is that $u=0$ holds, that is, the three sequences converge to $0$.
We may also assume that $\delta_n=(u_n-v_n)/(u_n-w_n)\to \delta$ with $|\delta|\le 1$
 by passing to a subsequence.

Note that $\Delta_n\to\infty$ enables us to assume as well that $u_n\ne 0$ for all $n$. Set $\ol v_n=v_n/u_n$ and $\ol w_n=w_n/u_n$.
By our assumptions,
$$
|1-\ol v_n|\le |1-\ol w_n|,\quad \delta_n=\frac{1-\ol v_n}{1-\ol w_n}\to\delta.
$$

If $(\ol w_n)_n$ tends to $\infty$, then $\ol v_n/\ol w_n\to \delta$ and
$$
\Delta_n=\frac{1-\ol v_n^k}{1-\ol w_n^k}=\frac{\ol w_n^{-k}-(\ol v_n/\ol w_n)^k}{\ol w_n^{-k}-1}\to \delta^k,
$$
a contradiction. So we may assume that   $(\ol w_n)_n$ is convergent towards a complex number $ \ol w$. It follows that  $(\ol v_n)_n$ is also
convergent, towards the complex number $\ol v = 1- \delta(1-\ol w)$.
Observe that
$$
\Delta_n=\frac{1-\ol v_n^k}{1-\ol w_n^k}=
\delta_n\frac{1+\ol v_n+\dots+\ol v_n^{k-1}}{1+\ol w_n+\dots+\ol w_n^{k-1}}.
$$
If $\ol w=1$ then $\ol v=1$ (because $|1-\ol v|\le |1-\ol w|$), and, then, we would get
$\Delta_n\to \delta$, a contradiction with $\Delta_n\to\infty$.
Therefore,  $\ol w$ is different from $1$.
Now $\Delta_n\to \infty$ implies that $\ol w (\ne 1)$ is a $k$-th root of unity.
 Since $w_n/u_n\to\ol w$, Proposition \ref{p:invt} implies it is impossible that $w_n$, $u_n$ are both in $R$ or both in $L$ for infinitely many values of $n$
 for it would contradict that these arcs are transverse to their rotated images under $\ol w$.
Thus we may assume that $w_n\in R$ and $u_n\in L$.
However, since $\ol w_n^k \to 1$, this would imply that $R'$ and $L'$ are not transverse, the final contradiction.
\end{proof}

\begin{thm}
  \label{t:qa}
Consider two simple arcs $R'$, $L'$ with common endpoint at $0$ and disjoint otherwise.
Suppose that $\lambda$ is a complex number with $|\lambda|>1$.
Furthermore, suppose that $\lambda R'\supset R'$ and $\lambda L'\supset L'$.
If $R'\cup L'$ is smooth except possibly at $0$, then $R'\cup L'$ is a quasi-arc.
\end{thm}

\begin{proof}
Assume the contrary: there are three sequences $x_n$, $y_n$, $z_n\in R'\cup L'$ such that
\begin{itemize}
  \item the point $y_n$ is always between $x_n$ and $z_n$ in the arc $R'\cup L'$
  (in particular, the three points $x_n$, $y_n$, $z_n$ are always different);
  \item we have $\delta_n=|x_n-z_n|/|x_n-y_n|\to 0$ as $n\to\infty$.
\end{itemize}
It follows from the second assumption that $|x_n-z_n|\to 0$ since the denominator is bounded.
We can now make a number of additional assumptions on $x_n$, $y_n$, $z_n$ by passing to subsequences.
Assume that $x_n$ and $z_n$ converge to the same limit.
If this limit is different from $0$, then straightforward geometric arguments yield a contradiction
(it is obvious that every closed subarc of $R'\cup L'$ not containing $0$ is a quasi-arc).
Thus we may assume that $x_n$, $z_n\to 0$.
Since $y_n$ is between $x_n$ and $z_n$, we also have $y_n\to 0$.
From now on, we rely on the assumption that all three sequences $x_n$, $y_n$, $z_n$ converge to $0$.
Assume that $x_n\ne 0$ for all $n$ (otherwise, for a suitable subsequence, $z_n\ne 0$ for all $n$, and we may interchange $x_n$ and $z_n$)
 and that $x_n\in R'$ rather than $L'$.

Take $r>0$ sufficiently small, and let $A$ be the annulus $\{z\in\C\,|\, r<|z|<|\lambda|r\}$.
For every $n$, there exists a positive integer $m_n$ such that $\lambda^{m_n}x_n\in\ol A$.
Set $x'_n$, $y'_n$, $z'_n$ to be $\lambda^{m_n}x_n$, $\lambda^{m_n}y_n$, $\lambda^{m_n}z_n$, respectively.
By the invariance property of $R'$, we must have $x'_n\in R'$.
Passing to a subsequence, arrange that $x'_n\to x\in\ol A$ as $n\to\infty$.
Since $|x'_n-z'_n|/|x'_n-y'_n|=\delta_n\to 0$, then
$z'_n\to x$.
Since the intersections of $R'$ and $L'$ with an open neighborhood of $\ol A$ are smooth open arcs,
 it follows that $z'_n\in R'$ for large $n$, hence $y'_n\to x$ and $\delta_n\not\to 0$.
A contradiction.
\end{proof}

\section{Reducing the admissible collection}
\label{s:rep}
Let $P$ and $\Zc$ be as in Theorem \ref{t:main-a}.
Let $c$ be a critical point of $P$.
A cut $\Ga=R\cup L\cup\{c\}$ formed by $c$ and two rays $R$, $L$ landing at $c$ such that $P(R)=P(L)$,
is a \emph{critical cut}. A pullback of a critical cut is a \emph{precritical cut}.

\subsection{Outline of the proof of Theorem \ref{t:main-a}}
The proof of Theorem \ref{t:main-a}, under the assumption that $P|_{A_P(\Zc)}$ is not injective, will go as follows.

\emph{Step 1: reduction}.
Take only the nondegenerate cuts from $\Zc$ with periodic root
points that are outward repelling.
These cuts define a polynomial-like Julia set by Theorem \ref{t:pl}.
Replacing $P$ with the straightening of the corresponding polynomial-like map
and using the conditions of Theorem \ref{t:main-a}, we may assume that
every periodic cut in $\Zc$ is degenerate and all cuts from $\Zc$
are %either
precritical cuts that eventually map to degenerate cuts with
repelling periodic root points, and their images.
Step 1 is made in this section.

\emph{Step 2: carrot modification}.
After the reduction step, we assume that every periodic cut in $\Zc$ is degenerate and has repelling root point.
Terminal segments of these degenerate cuts can be fattened to so called \emph{carrots}.
Carrots are quasi-disks; they are almost the same as ``sectors'' used in \cite{BD88}.
Moreover, if $C$ is a carrot corresponding to a periodic cut, then $C\cap K_P$ is only the root point of the cut.
Carrots corresponding to preperiodic cuts are defined differently,
 and their union contains $K_P\sm A_P(\Zc)$.
Finally, we modify $P$ in carrots and  near infinity and obtain a quasi-regular polynomial $f$.
Application of Theorem \ref{t:sw} to $f$ concludes the proof of the Main Theorem.
Step 2 is made in the next section.

If $P|_{A_P(\Zc)}$ is injective, then $A_P(\Zc)$ is a point, which is proved in Theorem \ref{t:1-1}
 by a different (but simpler) method.

\subsection{Eliminating some periodic cuts}
\label{ss:pc}
Let $\Zc$ be an admissible collection of cuts.
It is a union of a finite family of forward orbits of cuts.
Among them there might exist degenerate cuts whose backward orbit
in $\Zc$ does not include nondegenerate cuts. These cuts make no
impact upon the avoiding set $A_P(\Zc)$ (recall that it was defined in Definition \ref{d:avoid})
and will be called \emph{fictitious}.
Start by removing all fictitious cuts.
Thus, from now on, we assume that there are no fictitious cuts left.

\begin{dfn}\label{d:legal}
An admissible family of non-fictitious cuts is \emph{legal} if it is a union of
finite orbits of (pre)critical cuts each of which eventually maps to a degenerate cut
with repelling periodic root point.
\end{dfn}

Define $\Zc_{pc}\subset \Zc$ as the subset consisting of all periodic nondegenerate cuts $\Ga\in\Zc$.
(Here ``pc'' is from ``\textbf{p}eriodic \textbf{c}uts''.)
An admissible family of cuts is legal if $\Zc_{pc}=\0$, and all critical root points eventually map to repelling
periodic points.

Let $V_E$ be a closed Jordan disk around $K_P$ bounded by an equipotential curve.
A standard thickening of $V_E\sm \bigcup_{\Ga\in \Zc_{pc}} W_\Ga$ yields an open Jordan disk $V$
 and a polynomial-like map $P:U\to V$, where $U$ is the component of $P^{-1}(V)$ containing $A_P(\Zc)$.
The thickening construction first appeared in \cite{dou86} and, since then, was used in a variety of contexts 
 (see, e.g., \cite{hai00,kiw97,mil00}).
In our case, the statement follows from more general Lemma 5.13 of \cite{IK12}.
Below, we give a sketch of an argument in our setting, omitting some technical details.

The collection $\Zc_{pc}$ of cuts consists of one or several cycles under the action of $P$.
Choose one cycle of cuts $\Ga_0$, $\dots$, $\Ga_{m-1}$ in $\Zc_{pc}$, and let
 $W_0$, $\dots$, $W_{m-1}$ be the corresponding wedges.
The numbering can be arranged so that $P(\Ga_j)=\Ga_{j+1\pmod m}$ for $j=0$, $\dots$, $m-1$.
By our assumption, the root points $a_j$ of $W_j$ are repelling periodic points.
Let $p$ be the minimal period of $a_0$.
This number divides $m$ but may be smaller.

\begin{lem}
  \label{l:disks-hyp}
For $i=0$, $\dots$, $p-1$, there are disjoint Jordan disks $D_i$ around $a_i$ with the following properties.
The pullback of $D_{i+1\pmod p}$ containing $a_i$ is compactly contained in $D_i$.
The map $P$ is injective on each of $\ol D_i$, and $P(D_i)$ is disjoint from all $D_j$ with
 $j\ne i+1\pmod p$.
\end{lem}

\begin{proof}
Let $D_p$ be a sufficiently small round disk around $a_0$.
Then $P^p$ is injective on $D_p$, all $P^i(D_p)$ are disjoint for $i=0$, $\dots$, $p-1$, and $D_p\Subset P^p(D_p)$.
It follows that the $P^p$-pullback of $D_p$ containing $a_0$ is compactly contained in $D_p$.
For each $i=0$, $\dots$ $p-1$, consider the $P$-pullback of $D_{i+1}$ containing $a_i$;
 define $D_i$ as a tight neighborhood of this pullback.
Here a \emph{tight neighborhood} means a subneighborhood in the $\eps$-neighborhood for sufficiently small $\eps>0$.
Then we have $D_0\Subset D_p$ provided that all chosen neighborhoods were sufficiently tight.
Since a pullback of $D_p$ is compactly contained in $D_{p-1}$, a pullback of $D_0$ is so as well.
The fact that $D_i$ are disjoint and the very last claim both follow from out assumption that $P^i(D_p)$ are disjoint.
\end{proof}

The standard thickening $V$ of the continuum
$$
V' = V_E\sm\bigcup_{\Ga\in\Zc_{pc}} W_\Ga
$$
(recall that $V_E$ is closed) is now obtained as a tight Jordan neighborhood of $V'$, whose boundary
 consists of: arcs of some equipotential close to $E$, segments of external rays in $W_\Ga$
 close to $\Ga$, where $\Ga\in\Zc_{pc}$, boundary arcs of the disks constructed in Lemma \ref{l:disks-hyp}
 (or perhaps smaller disks with the same property but having simplest possible intersections with the rays).
Let $U$ be the $P$-pullback of $V$ containing $\Ac_P(\Zc)$.
Then $P:U\to V$ is the desired PL map.
Indeed, each piece of the boundary $\bd(V)$ has its pullbacks inside $V$, by definition,
 which implies that $U\Subset V$.

By Theorem \ref{t:DH-pl},
the PL-map $P:U\to V$ is hybrid equivalent (by a map $\psi$) to $\tilde P:\tilde U\to \tilde V$,
 where $\tilde P$ is a polynomial and $\tilde U$ is a neighborhood of its filled Julia set.

 \begin{lem}
  \label{l:rem-pc}
There is a legal family $\tilde\Zc$ of cuts in the dynamical plane of $\tilde P$
such that $A_{\tilde P}(\tilde\Zc)=\psi(A_P(\Zc))$.
\end{lem}

Lemma \ref{l:rem-pc} is left to the reader.
If Theorem \ref{t:main-a} is proved for $\tilde P$ and $A_{\tilde P}(\tilde\Zc)$, then it would follow for $A_P(\Zc)$.
This reduces the Main Theorem to the case of legal families of cuts.

\section{Carrots}
\label{s:plan}
From now on we assume that $\Zc$ is a legal family of cuts, with no fictitious cuts.

\begin{figure}
  \centering
  \includegraphics[width=5cm]{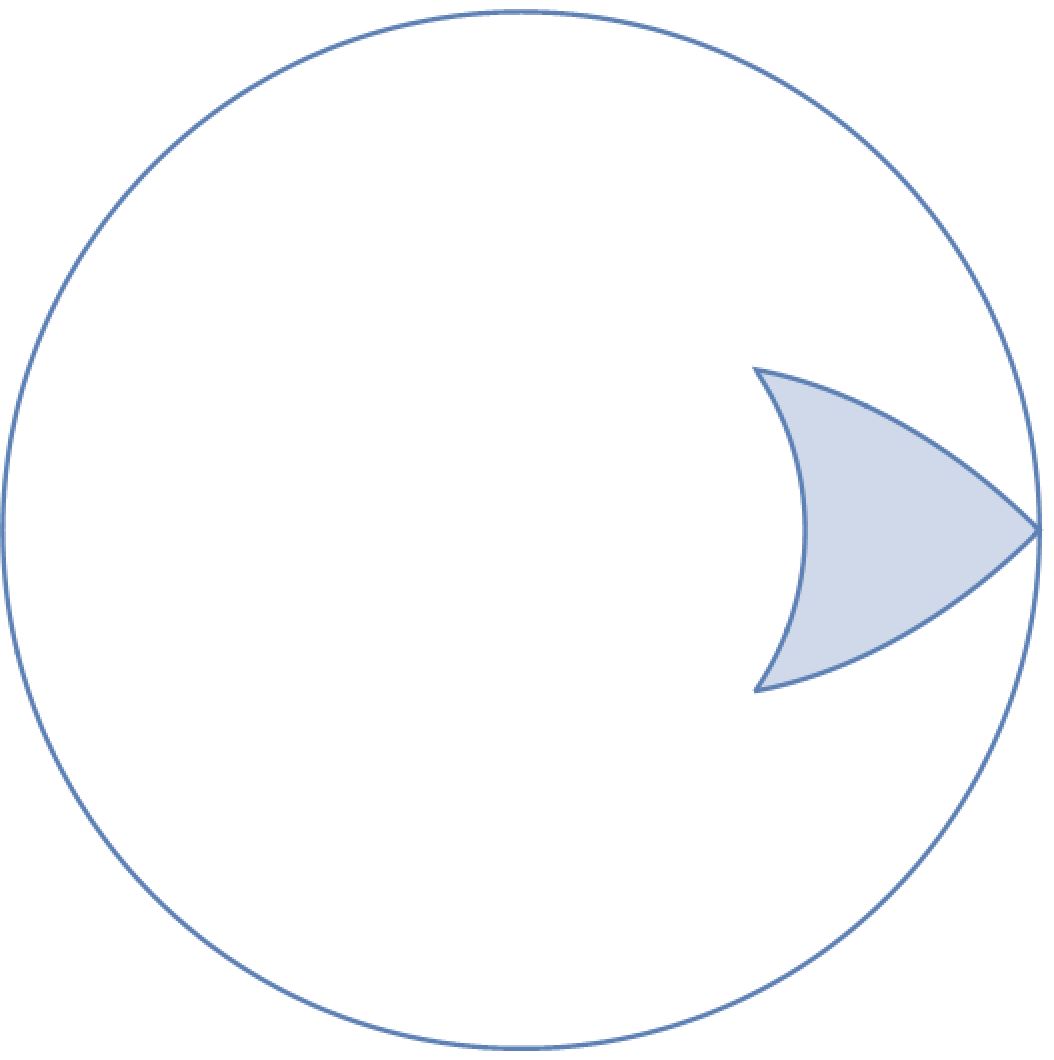}
  \caption{A proto-carrot.}\label{fig:pcarrot}
\end{figure}

\begin{dfn}[Prototype carrot]
The \emph{prototype carrot}, or simply \emph{proto-carrot} $\c(\rho_0,\ta_0)$, with polar parameters $(\rho_0,\ta_0)$
 is the ``triangular'' region in $\ol\disk$ given by the inequalities $\rho_0\le\rho\le e^{-|\ta-\ta_0|}$
 in the polar coordinates $(\ta,\rho)$.
Here $\ta$ is an \emph{angular coordinate} so that $\ta-\ta_0$ can be either positive or negative,
and $\rho$ is the \emph{radial coordinate}, i.e., the distance to the origin.
We assume that the parameter $\rho_0<1$ is close to $1$.
A proto-carrot is bounded by a circle arc and two symmetric segments of logarithmic spirals, see Figure \ref{fig:pcarrot}.
All proto-carrots are homeomorphic.
\end{dfn}

A part of the boundary of $\c(\rho_0,\ta_0)$ near point $e^{2\pi i\ta_0}$ is given by $\rho=e^{-|\ta-\ta_0|}$
and consists of two analytic curves meeting at $e^{2\pi i\ta_0}$ and invariant under the map $z\mapsto z^d$ (regardless of $d$).
Since rotation by $\ta_0$ composed with $z^d$ equals $z^d$ composed with the rotation by $d\cdot \ta_0$,
 the next proposition follows.

\begin{prop}
  \label{p:sid-td}
 Take any $\ta_0\in\R/\Z$.
The map $z\mapsto z^d$ takes $\c(\rho_0,\ta_0)$ to $\c(\rho_0^d,d\ta_0)$, provided that $\rho_0$ is close to $1$.
\end{prop}

Let us recall the following concept.

\begin{dfn}[Stolz angle]
A \emph{Stolz angle} in $\disk$ at a point $u\in\uc$ is by definition a convex cone with apex at $u$
  bisected by the radius and with aperture strictly less than $\pi$.
\end{dfn}

The proto-carrot $\c(\rho,\ta_0)$ approaches the unit circle within some Stolz angle at $e^{2\pi i\ta_0}$
(the Implicit Function Theorem shows that the aperture of such Stolz angles can be made arbitrarily close to $\pi/2$).
The following theorem (see, e.g.,\cite[Theorem 2.2]{CG}) describes an important property of Stolz angles.

\begin{thm}
  \label{t:stolz}
Consider a simply connected domain $D\subset\ol\C$ that is not the sphere minus a singleton,
 and let $\psi:\disk\to D$ be a conformal isomorphism.
Suppose that a point $z_0\in\bd(D)$ is accessible from $D$.
Then there is a point $u_0\in\uc$ with the following property:
 $\psi(u)\to z_0$ as $u\to u_0$ inside any Stolz angle with apex at $u_0$.
\end{thm}

The point $u_0$ from Theorem \ref{t:stolz} is uniquely defined once an access to $z_0$ from $D$ is chosen.
Let $E(\rho)$ be the equipotential $\psi_P(\{z\in\disk\mid |z|=\rho\})$.
Write $U(\rho)$ for the open Jordan domain bounded by $E(\rho)$.

\begin{dfn}[Carrots]
  \label{d:carrot}
Let $\Ga\in \Zc$ be a periodic degenerate cut with a repelling root point $z_\Ga$.
Let $\ta_\Ga\in \D$ be the angle corresponding to $\Ga$.
Define the \emph{carrot} $C_\Ga(\rho)$ as $\psi_P(\c(\rho,\ta_\Ga))$.
This is the closed triangular region bounded by three arcs $R_\Ga(\rho)$, $L_\Ga(\rho)$, and $E_\Ga(\rho)$.
Here $R_\Ga(\rho)$ and $L_\Ga(\rho)$ are simple topological arcs landing at $z_\Ga$;
 the latter follows from Theorem \ref{t:stolz}, in which the access to $z_\Ga$ from the basin of infinity
 is defined by a terminal segment of the ray $\Ga$.
The arc $E_\Ga(\rho)$ is a part of the equipotential curve $E(\rho)$.

It remains to define carrots for strictly preperiodic cuts in $\Zc$.
Take such a cut $\Ga\in \Zc$; we may assume by induction that the carrot $C_{P(\Ga)}(\rho^d)$ is already defined.
Let $W$ be the wedge corresponding to $\Ga$.
There is a unique pullback $R_\Ga(\rho)$ of $R_{P(\Ga)}(\rho^d)$ and a unique pullback $L_\Ga(\rho)$ of $L_{P(\Ga)}(\rho^d)$
with the following properties:
\begin{enumerate}
  \item both $R_\Ga(\rho)$ and $L_\Ga(\rho)$ land at $z_\Ga$, the root point of $\Ga$;
  \item the arc $E_\Ga(\rho)$ of $E(\rho)$ with endpoints $R_\Ga(\rho)\cap E(\rho)$ and $L_\Ga(\rho)\cap E(\rho)$
  is the smallest subarc of $E(\rho)$ which contains $W\cup \Ga$ and has endpoints in $P^{-1}(R_{P(\Ga)}(\rho^d)$
  and $P^{-1}(L_{P(\Ga)}(\rho^d))$.
\end{enumerate}
The set $C_\Ga(\rho)$ is then defined as a triangular region bounded by three arcs $R_\Ga(\rho)$, $L_\Ga(\rho)$ and $E_\Ga(\rho)$.
\end{dfn}

Schematic Figure \ref{fig:car} illustrates the definition of a carrot in the case when $z$ is critical.

\begin{figure}
  \centering
  \includegraphics[width=10cm]{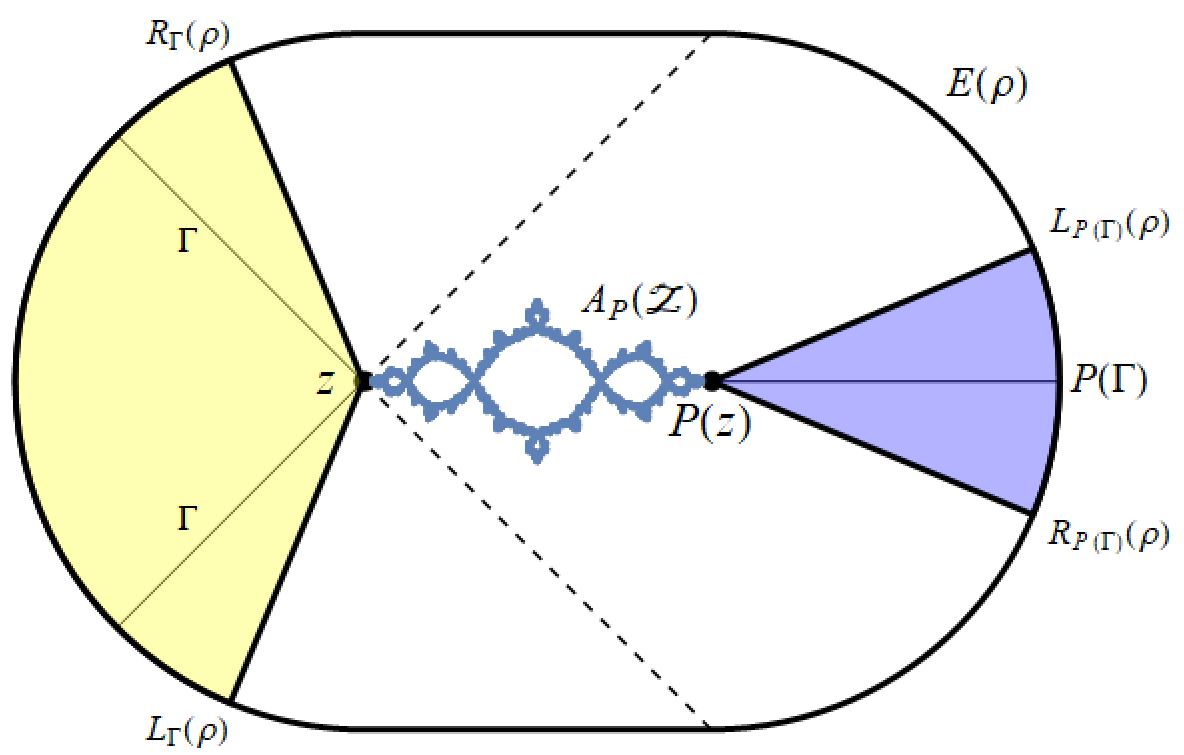}
  \caption{Definition of a carrot $C_\Ga(\rho)$, where $z=z_\Ga$ is critical.
  The carrot $C_{P(\Ga)}(\rho)$ is the (dark) shaded sector on the right and $C_\Ga(\rho)$ is the (light) shaded sector on the left.
  Here $P(\Ga)$ is degenerate, the local degree of $P$ at $z$ is 4, and the set $P^{-1}(P(\Ga))\sm\Ga$ consists of two dashed lines.
  }\label{fig:car}
\end{figure}

\subsection{Carrots are quasi-disks}
\label{ss:qd}
We will need the following geometric property of carrots.

\begin{prop}
\label{p:qd}
Let $\Ga\in \Zc$ have periodic repelling root point $z_\Ga$.
Then $C_\Ga(\rho)$ is a quasi-disk, for every $\rho$ sufficiently close to $1$.
\end{prop}

\begin{proof}
By our construction, it is enough to prove that $R_\Ga(\rho)\cup L_\Ga(\rho)$ is a quasi-arc locally near $z_\Ga$.
Observe that this property is independent of $\rho$.
Consider a local holomorphic coordinate $u$ near $z_\Ga$ such that $u=0$ at $z_\Ga$,
 and $P^m$ takes the form $u\mapsto\lambda u$.
Here $\lambda$ is the derivative of $P^m$ at $z_\Ga$, hence $|\lambda|>1$.
A local coordinate $u$ with the properties stated above exists by the classical K\"onigs linearization theorem.
Set $R'$, $L'$ to be the images of $R_\Ga(\rho)$, $L_\Ga(\rho)$ in the $u$-plane and apply Theorem \ref{t:qa} to $R'$, $L'$.
Since a holomorphic local coordinate change takes quasi-arcs to quasi-arcs, we obtain the desired.
\end{proof}

\begin{lem}
  \label{l:cqd}
All carrots $C_\Ga(\rho)$ with $\Ga\in \Zc$ and $\rho\in (0,1)$ sufficiently close to $1$ are quasi-disks.
Moreover, the map
$
P:R_\Ga(\rho)\cup L_\Ga(\rho)\to R_{P(\Ga)}(\rho^d)\cup L_{P(\Ga)}(\rho^d)
$
is quasi-symmetric.
\end{lem}

\begin{proof}
  Take $\Ga\in \Zc$ with root point $z=z_\Ga$.
If $z$ is periodic, then it is repelling,
and the map $P:\bd(C_\Ga(\rho))\to \bd(C_{P(\Ga)}(\rho^d))$ is quasi-symmetric
 (note that $P$ is conformal on a neighborhood of the boundary of $C_\Ga(\rho)$).
In this case, $C_\Ga(\rho)$ is a quasi-disk by Proposition \ref{p:qd}.
Suppose now that $z$ is strictly preperiodic.
We may assume by induction that $C_{P(\Ga)}(\rho^d)$ is a quasi-disk.
If $z$ is not critical, then it follows immediately that $C_\Ga(\rho)$ is also a quasi-disk,
 as a pullback of $C_{P(\Ga)}(\rho^d)$ under a map that is one-to-one and conformal in a neighborhood of $\bd(C_\Ga(\rho))$.
It also follows that the map $P:\bd(C_\Ga(\rho))\to \bd(C_{P(\Ga)}(\rho^d))$ is quasi-symmetric.

Finally, assume that $z$ is critical. Let $n$ be the smallest positive integer with $P^n(z)$ periodic,
and let $m$ be the minimal period of $P^n(z)$.
It suffices to prove that
$$
P^n:R_\Ga(\rho)\cup L_\Ga(\rho)\to R_{P^n(\Ga)}(\rho^{d^n})\cup L_{P^n(\Ga)}(\rho^{d^n})
$$
is QS.
Indeed, we may choose local coordinates $x$ near $z$ and $y$ near $P^n(z)$ so that the map $P^n$ takes the form $y=x^k$
 for some integer $k>1$ (this integer is the local degree of $P^n$ at $z$).
Moreover, the $y$ coordinate can be chosen so that $P^m$ takes the form $y\mapsto \la y$.
In these coordinates, Lemma \ref{l:RL-qs} applies and yields the desired.
\end{proof}

\subsection{The carrot modification of $P$}
Choose $\rho\in (0,1)$ close to $1$ so that the carrots $C_\Ga(\rho)$
for degenerate cuts $\Ga\in\Zc$ are disjoint except, possibly, for root points.
Recall that $U(\rho)$ is the bounded component of $\C\sm E(\rho)$.
Then $K_P\subset U(\rho)$.
In this section, we modify $P$ to form a new map $P^c:\ol U(\rho)\to \ol U(\rho^d)$.
First, let us define $P^c$ so that
$P^c=P$ outside of $\bigcup_{\Ga\in\Zc_{cr}} C_\Ga(\rho)$, where $\Zc_{cr}$ is the set of all \emph{critical} cuts in $\Zc$,
i.e., cuts with critical root points.

Suppose now that $\Ga\in\Zc_{cr}$.
Set $P^c=P$ on $R_\Ga(\rho)\cup L_\Ga(\rho)$.
Note that $E_\Ga(\rho)$ wraps around the entire $E(\rho^d)$ under $P$.
Define $P^c$ on $E_\Ga(\rho)$ as a QS isomorphism between $E_\Ga(\rho)$ and $E_{P(\Ga)}(\rho^d)$.
Thus, by the remark made above, $P^c$ is necessarily different from $P$ on $E_\Ga(\rho)$.
Finally, let $P^c:C_\Ga(\rho)\to C_{P(\Ga)}(\rho^d)$ be a QS map that extends the already defined map
 $P^c:\bd(C_\Ga(\rho))\to \bd(C_{P(\Ga)}(\rho^d))$.
The existence of such extension is guaranteed by Theorem \ref{t:qs-ext}.
The map $P^c:\ol U(\rho)\to\ol U(\rho^d)$ is a \emph{carrot modification} of $P$.
Clearly, $P^c:U(\rho)\to U(\rho^d)$ is a proper map; let $d_c$ be its topological degree.
Observe that $d_c<d$ provided that $\Zc_{cr}\ne\0$.

Recall that quasi-regular polynomials were introduced in Definition \ref{d:qr-poly}.

\begin{lem}
  \label{l:f}
  There is a quasi-regular degree $d_c$ polynomial
  $f:\C\to\C$ such that $f=P^c$ on $\ol U(\rho)$ and
 $P^c=P$ outside of $\bigcup_{\Ga\in\Zc_{cr}} C_\Ga(\rho)$.
In particular, $f$ is holomorphic on a neighborhood of infinity.
\end{lem}

\begin{proof}
The map $P^c: U(\rho)\to U(\rho^d)$ is obtained by gluing together finitely many quasi-regular maps along quasi-arcs.
Such map is itself quasi-regular, as follows from the ``QC removability'' of quasi-arcs, cf. Proposition 4.9.9 of \cite{hub}.

  Set $\disk(\rho)=\{z\in\disk\mid |z|<\rho\}$; this is the disk of radius $\rho$ around $0$.
Clearly, there is a quasi-regular map $g:\ol\disk(\rho)\to\ol\disk(\rho^d)$ such that $g(z)=z^{d_c}$ in a neighborhood of $0$,
and $g=\psi_P^{-1}\circ P^c\circ\psi_P$ on the boundary of $\disk(\rho)$.
It suffices to define $f$ as $P^c$ on $\ol U(\rho)$ and as $\psi_P\circ g\circ\psi_P^{-1}$ on $\C\sm\ol U(\rho)$.
\end{proof}

Define a positive integer $T_0$ so that $\ol \d f=0$ on the unbounded complementary component of $E(\rho^{d^{T_0}})$.
Let $A_0$ be a topological annulus such that the bounded complementary component of $A_0$ lies in $\ol U(\rho)$,
 and $\ol\d f=0$ in the unbounded complementary component.
We may assume that $A_0$ is bounded by $E(\rho)$ from the inner side and by $E(\rho^{d^{T_0}})$ from the outer side.
Set $A=A_0\cup A_{cr}$, where $A_{cr}=\bigcup_{\Ga\in \Zc_{cr}} C_\Ga(\rho)$.
Then by definition $\ol\d f=0$ outside of $A$.
In order to verify the assumptions of Theorem \ref{t:sw}, it remains to prove the following lemma.

\begin{lem}
  \label{l:T}
  Define $T_{cr}$ as the cardinality of $\Zc_{cr}$; set $T=T_{cr}+T_0$.
  The forward $f$-orbit of any point $x\in\C$ can visit $A_{cr}$ at most $T_{cr}$ times.
Therefore, it can visit $A$ at most $T$ times.
\end{lem}

\begin{proof}
It suffices to prove the first statement.
Define the subset $\mathsf{X}\subset\disk$ consisting of all points, whose polar coordinates $(\rho,\ta)$ satisfy
$$
\rho\le e^{-|\ta-\ta_\Ga|}
$$
for at least one periodic $\Ga\in\Zc$ (then $\Ga$ is necessarily degenerate by our assumption on $\Zc$).
Clearly, $\mathsf{X}$ is forward invariant under the map $(\rho,\ta)\mapsto (\rho^d,d\ta)$.
Moreover, $\psi_P(\mathsf{X})$ includes all $C_\Ga(\rho')$ for all periodic $\Ga\in\Zc$ and all $\rho'\in (0,1)$.

Suppose now that $\rho_0\in (0,1)$ is sufficiently close to $1$.
A forward $P$-orbit of a point $x\in A_{cr}$ may visit $A_{cr}$ at most $T_{cr}$ times before it first enters $\psi_P(\mathsf{X})$
(that is, it may visit each $C_\Ga(\rho_0)$ with $\Ga\in\Zc_{cr}$ at most once).
Thus it suffices to prove that no point of $\psi_P(\mathsf{X})$ can map to $C_\Ga(\rho_0)$ with $\Ga\in\Zc_{cr}$
 under an iterate of $f$.
Since $\psi_P(\mathsf{X})$ is forward invariant, it suffices to choose $\rho_0$ so that
 $C_\Ga(\rho_0)\cap\psi_P(\mathsf{X})=\0$ for all $\Ga\in\Zc_{cr}$.

Take any $\Ga\in\Zc_{cr}$.
The set of all angles $\ta$ such that $R_P(\ta)$ is separated from $A_P(\Zc)$ by $\Ga$ is an arc $I_\Ga$ of $\R/\Z$
 whose length is an integer multiple of $1/d$
 (indeed, the endpoints of this arc are mapped to the same point under the $d$-tupling map).
Define
$$
\c(\rho_0,I_\Ga)=\{(\rho,\ta)\mid \exists\ta_0\in I_\Ga\quad \rho_0\le\rho\le e^{-|\ta-\ta_0|}\}.
$$
Then all points in $C_\Ga(\rho_0)\sm K_P$ are necessarily in $\psi_P(\c(\rho_0,I_\Ga))$.
It is clear that no $\ta_\Ga$ with periodic $\Ga$ can belong to $I_\Ga$.
Therefore, $\c(\rho_0,I_\Ga)$ is disjoint from $\mathsf{X}$ for $\rho_0$ sufficiently close to $1$.
It follows that $C_\Ga(\rho_0)\cap\psi_P(\mathsf{X})=\0$, as desired.
\end{proof}

The set $A_P(\Zc)$ is a fully invariant set for $f$.
We now assume that the map $P:A_P(\Zc)\to A_P(\Zc)$ is not injective, hence $d_c\ge 2$.
Thus all assumptions of Theorem \ref{t:sw} are fulfilled.
Then there is a QC map $\Psi:\ol\C\to\ol\C$ and a rational map $Q:\C\to\C$ of degree $d_f$ such that $Q\circ\Psi=\Psi\circ f$.
It can be arranged that $\Psi(\infty)=\infty$.
With this normalization, $Q^{-1}(\infty)=\{\infty\}$, therefore, $Q$ is a degree $d_c$ polynomial.

\begin{thm}
  \label{t:P^c-APL}
  The set $\Psi(A_P(\Zc))$ coincides with $K_Q$.
\end{thm}

\begin{proof}
Since $A_P(\Zc)$ is $P$-forward invariant, $\Psi(A_P(\Zc))\subset K_Q$.
It remains to prove that any point $y=\Psi(x)$ with $x\notin A_P(\Zc)$ escapes to infinity under the iterations of $Q$.
Equivalently, $x$ escapes to infinity under the iterations of $f$.
Indeed, if the forward $f$-orbit of $x$ is outside of $K_P$ and outside of all carrots, then $f^n(x)=P^n(x)\to\infty$.
If $x$ is in $K_P$ but not in $A_P(\Zc)$, then $f^k(x)\in C_\Ga(\rho)$ for some $k\ge 0$ and some $\Ga\in\Zc$.
Possibly replacing $\rho$ with $\rho^{d^l}$ with a suitable $l$ and $x$ with $f^l(x)$, we may assume that $\Ga$ is periodic.
However, in this case $C_\Ga(\rho)$ is in the $P$-basin of infinity, hence $f^n(x)=P^{n-k}\circ f^k(x)\to\infty$.
\end{proof}

Since $\deg(Q)=\deg(f)=d_c<\deg(P)$, Theorem \ref{t:main-a} is proved in the case when $P|_{A_P(\Zc)}$ is not injective.

\subsection{The case when $P$ is injective on $A_P(\Zc)$}
Theorem \ref{t:1-1} completes the proof of the Theorem \ref{t:main-a}.

\begin{thm}
  \label{t:1-1}
  Suppose that all assumptions of the Main Theorem are fulfilled and $P:A_P(\Zc)\to A_P(\Zc)$ is one-to-one.
Then $A_P(\Zc)$ is a single repelling point.
\end{thm}

\begin{proof}
  Replace $P$ with a suitable iterate to arrange that all periodic cuts in $\Zc$ are fixed.
Let $\Ga$ be such a fixed cut, and $z_\Ga$ its root point.
Then $P(z_\Ga)=z_\Ga$.
We claim that there are no critical points of $P$ in $A_P(\Zc)\sm\rt_\Zc$.
Indeed, consider a critical point $c\in A_P(\Zc)$.
A point $w\in A_P(\Zc)$ near $P(c)$ has at least two preimages $z$, $z'$ near $c$.
If both are in $\C\sm\bigcup\Wc_\Zc$, then both are in $A_P(\Zc)$, a contradiction.
Thus, say, $z'$ is not in the principal set; then it must be separated from $A_P(\Zc)$ by a cut from $\Zc$.
Since $w$ can be chosen arbitrarily close to $P(c)$, the point $c$ itself must belong to $\rt_\Zc$.

Suppose that $A_P(\Zc)$ is not a singleton.
Then Theorem 7.4.7 of \cite{bfmot13} is applicable to the $P$-invariant continuum $A_P(\Zc)$.
This theorem states that there is a rotational fixed point in $A_P(\Zc)$.
That is, either a non-repelling fixed point $a$ or a repelling fixed point $a$ such that the external rays of $P$ landing at $a$
 undergo a nontrivial combinatorial rotation.
If $a$ is non-repelling, then there is a critical point $c$ that is not
preperiodic and not separated from $a$ by $\Zc$. In the attracting
and parabolic cases this follows from classical results of Fatou \cite{Fat20}.
Suppose that $a$ is a Cremer of Siegel fixed point. This case was considered in Theorem 4.3
\cite{bclos16} (the proof is based upon \cite{bm05} and classical results of Ma\~n\'e \cite{Man93})
that implies that then $A_P(\Zc)$ must contain a recurrent critical point.
By the previous paragraph this leads to a contradiction.
Thus $a$ is repelling and rotational. However,
since $P|_{A_P(\Zc)}$ is one-to-one, this implies that there are no
other fixed points in $A_P(\Zc)$. Therefore, $a=z_\Ga$; but the latter
is non-rotational, a contradiction. We conclude that $A_P(\Zc)$ is a
singleton.
\end{proof}

\section{Proof of the Main Theorem}
\label{s:fk-apl}
Consider a continuum $Y\subset K_P$ such that $P:Y\to Y$ is a degree $k$ branched covering.
By definition, there are open neighborhoods $U$ and $V$ of $Y$ and a degree $k$ branched covering $\tilde P:U\to V$
 such that $P=\tilde P$ on $Y$ and $\tilde{P}^{-1}(Y)=Y$.
For every $y\in Y$, the local multiplicity $\mu_Y(y)$ is defined as the multiplicity of $y$ with respect to $\tilde P$.
For all $z\ne P(y)$ very close to $P(y)$, exactly $\mu_Y(y)$ points of $P^{-1}(z)\cap Y$ are near $y$.
If $y$ is not critical then $\mu_Y(y)=1$.
On the other hand, some critical points of $P$ in $Y$ may also have multiplicity $1$ with respect to $Y$
 (that is, with respect to $\tilde P$).
Irregular points of $Y$ are precisely the points $y\in Y$ with $\mu_{K_P}(y)>\mu_Y(y)$.

The proof of the Main Theorem splits into several steps.

\subsection{Reduction to the case when $K_P$ is connected}
Let $P$ and $Y$ be as above.
Suppose first that $K_P$ is disconnected.
Let $K_P(Y)$ be the component of $K_P$ containing $Y$.
Clearly, $K_P(Y)$ is a $P$-invariant continuum.
Choose a tight equipotential $E_{V^*}$ around $K_P(Y)$ so that
the disk $V^*$ bounded by $E_{V^*}$ does not contain escaping critical points of $P$.
Then $P:U^*\to V^*$ is a PL map with filled Julia set $K_P(Y)$, where $U^*$ is the component of $P^{-1}(V^*)$ containing $Y$.
By Theorem \ref{t:DH-pl}, the PL map $P:U^*\to V^*$ is \emph{hybrid equivalent} to a PL restriction of a polynomial, say, $P^*$.
Let $Y^*$ be the subset of $K_{P^*}$ corresponding to $Y\subset K_P(Y)$.
Evidently, $P^*$ and $Y^*$ satisfy the assumptions of the Main Theorem.
Thus, we can consider only polynomials with connected Julia sets.

\subsection{Defining an admissible collection of cuts}
From now on, assume that the Julia set of $P$ is connected.
Start by defining a collection of cuts whose root points are irregular points of $Y$.
Let $a$ be an irregular point; it is necessarily a critical point of $P$.
By the assumptions of the Main Theorem, the point $a$ is eventually mapped to a repelling periodic point.
It follows from the Landing Theorem that there are (pre)periodic external rays landing at $a$.
Recall that a cut $\Ga=R\cup L\cup\{a\}$ formed by $a$ and two rays $R$, $L$ landing at $a$ such that $P(R)=P(L)$ is a \emph{critical cut}.
The corresponding wedge $W$ is called a \emph{critical wedge} at $a$.
A critical wedge $W$ at $a$ is $Y$-\emph{empty} if $W\cap Y=\0$.

\begin{lem}
  \label{l:W0}
  Suppose that $\mu_Y(a)<\mu_{K_P}(a)$.
Then there is at least one $Y$-empty critical wedge $W$ at $a$.
\end{lem}

\begin{proof}
Consider all components of $Y\sm\{P(a)\}$; let $s$ be the number of them.
By Theorem 6.6 of \cite{mcm94}, all components of $K_P\sm \{P(a)\}$ are separated from each other
 by the star cut formed by $P(a)$ and all external rays landing at $P(a)$.
By the main result of \cite{bot21}, each component of $K_P\sm\{P(a)\}$ includes at most one component of $Y\sm\{P(a)\}$.
Therefore, every component of $Y\sm\{P(a)\}$ is separated from the next one in
the cyclic order by an external ray landing at $P(a)$. Denoting components of
$Y\sm \{P(a)\}$ by $Y_1, \dots, Y_s$ and external rays landing at $a$ and separating these components by $R_1, \dots, R_s$ we may assume that
$$Y_1\prec R_1\prec Y_2\prec R_2\prec \dots\prec Y_s\prec R_s$$
where $\prec$ indicates positive (counterclockwise) circular direction.

If we pull this picture back to $a$ we will see that there are $\mu_{K_P}(a)$ pullbacks of each ray $R_i$ and
 $\mu_{K_P}(a)$ pullbacks of each set $Y_j$ from the previous paragraph ``growing'' out of $a$.
Since $\mu_Y(a)<\mu_{K_P}(a)$, not all pullbacks of sets $Y_i$ are contained in $Y$,
 some of them are \emph{not} contained in $Y$.
However, it follows from the definitions, in particular, from the fact that $P|_Y$ coincides with $\tilde{P}|_Y$,
 that the circular order of the pullbacks of $Y_i$ contained in $Y$ must follow that of the sets $Y_1$, $\dots$, $Y_s$.
Let us now choose a pullback of $Y_1$ that \emph{is} contained in $Y$, and move from it in the positive direction.
We will  be encountering pullbacks of sets $Y_i$ and pullbacks of rays $R_j$ in the same order of increasing of
 their subscripts until we reach the next pullback of $Y_1$.
However, $\mu_Y(a)<\mu_{K_P}(a)$.
Hence at some moment in this process the pullback $Y'_i\subset Y$ of $Y_i$ and the following it
 pullback $Y'_{i+1}\subset Y$ of $Y_{i+1}$ are not located in the adjacent pullbacks of the wedges between
 the corresponding external rays.
Rather, there will be a pullback $R'_i$ of $R_i$ and then the next
(in the sense of positive circular order) pullback $R''_i$ of $R_i$ such that there are no points of $Y$ in between these rays.
The wedge between $R'_i$ and $R''_i$ is the desired $Y$-empty critical wedge $W$ at $a$.
\end{proof}

Define $\Zc^{irr}$ (``irr'' stands for ``irregular'') as the set of boundary cuts of all $Y$-empty critical wedges
 at all irregular points of $Y$.
More precisely, for every irregular point $a\in Y$, mark specific $s$ rays separating $s$ components of $Y\sm\{P(a)\}$.
Then choose all $Y$-empty critical wedges at $a$ bounded by pullbacks of the marked rays (cf. Lemma \ref{l:W0}).
The family of cuts $\Zc^{irr}$ is clearly admissible.

\subsection{Reducing to the case of no irregular points}
We keep the notation introduced above.
By definition of $\Zc^{irr}$, we have $Y\subset A_P(\Zc^{irr})$.
By Theorem \ref{t:main-a} applied to $P$ and $\Zc^{irr}$, there is a polynomial $P^*$ such that
 $P^*:K_{P^*}\to K_{P^*}$ is topologically conjugate to $P:A_P(\Zc^{irr})\to A_P(\Zc^{irr})$.
Let $Y^*$ be the $P^*$-invariant continuum corresponding to $Y$ under this conjugacy.
We claim that $Y^*$ contains no irregular points.

If $a\in Y^*$ is an irregular point, then $\mu_{Y^*}(a)<\mu_{K_{P^*}}(a)$.
By Lemma \ref{l:W0}, there is a $Y^*$-empty critical wedge at $a$.
A corresponding $Y$-empty critical wedge at $a$ must be included into $\Zc^{irr}$; a contradiction.
Thus all points of $Y^*$ are regular.

Replacing $P$ with $P^*$ and $Y$ with $Y^*$, we may now assume that
 $P:Y\to Y$ has no irregular points.
However, then $P:Y\to Y$ satisfies the assumptions of Theorem \ref{t:poly-like}.
(Observe that the absence of irregular points is equivalent to the condition that $Y$ is a component of $P^{-1}(Y)$.)
The conclusion of the Main Theorem now follows from Theorem \ref{t:poly-like}.

\subsection{Proofs of Corollaries \ref{c:main1} and \ref{c:dou}}
Finally, we prove Corollaries stated in the introduction.

\begin{proof}[Proof of Corollary \ref{c:main1}]
Assume that $F$ is an invariant planar fiber of $P$ such that $Y=F\cap K_P$ is not a singleton.
%$Y$ is an invariant non-degenerate fiber of $P$.
We claim that $P|_Y$ is a degree $k$ branched covering for some $k>1$.
It is easy to see that planar fibers map onto (and locally onto) planar fibers
(see, e.g.,  \cite{sch99} or \cite{bclos16}); in particular, $P(Y)=Y$.
Observe that if $P|_Y$ is 1-to-1 then all the arguments of Theorem \ref{t:1-1} apply to $Y$
and imply that $Y$ is a singleton, a contradiction.
Hence there are points of $Y$ with more than one preimage in $Y$.

Suppose that $z\in Y$ is irregular. Then $z$ is critical,
and there are pairs of points $y'$, $y''$ arbitrarily close to $z$ such that $P(y')=P(y'')=y$, where $y'\in Y$ and $y''\notin Y$.
We claim that then $z$ is preperiodic, and there are several (rational) rays
that land at $z$. Suppose otherwise.
Choose a rational cut $\Ga''$ that separates $z$ and $y''$; set $\Ga=P(\Ga'')$.
By the assumption, $\Ga''$ does not contain $z$.
But then there exists another cut $\Ga'\subset P^{-1}(\Ga)$ that separates $y'$ from $z$, a contradiction.
Hence $z$ is preperiodic and there are rational rays landing at $z$.
This implies that $z$ maps to a repelling periodic point (recall that by the assumptions of Corollary
\ref{c:main1} there are no post-critical parabolic points in $Y$) and fulfills one of the assumptions of the Main Theorem.
 %\ref{t:fk-apl}.

It remains to verify that $P:Y\to Y$ is a covering.
Rather than doing this directly, it will be easier to use Theorem \ref{t:main-a}.
By definition of a fiber, there is a wedge $W_z$ at $z$ such that $\bd(W_z)=\Ga_z$ is a critical cut and $Y\subset\ol W$.
We may also assume that the component of $P^{-1}(P(\Ga_z))$ containing $z$ is disjoint from $W_z$.
Consider the collection $\Zc^{irr}=\{\Ga_z\}$, where $z$ runs through the set of all irregular points of $Y$.
Clearly, $\Zc^{irr}$ is admissible and satisfies the assumptions of Theorem \ref{t:main-a}.
It follows that the corresponding avoiding set $A=A_P(\Zc^{irr})\supset Y$ gives rise to a polynomial $Q$ such that
 $Q:K_Q\to K_Q$ is topologically conjugate to $P:A\to A$.
Moreover, the conjugacy extends as a positively oriented homeomorphism between neighborhoods of $K_Q$ and $A$.
(This extension is not a conjugacy, however.)
Passing from $P$ to $Q$, we may assume that $Y$ has no irregular points at all.
In this case, $P:U\to V$ is a degree $k$ covering for some $k>1$ and some neighborhoods $U$, $V$ of $Y$.
Moreover, $Y=P^{-1}(Y)\cap U$.
Thus, $Y$ satisfies all assumptions of the Main Theorem, and we are done.
\end{proof}

\begin{proof}[Proof of Corollary \ref{c:dou}]
Suppose that $P$ is a cubic polynomial such that $P(0)=0$ and $P'(0)=e^{2\pi i\theta}$ with $\theta\in\R\sm\Q$.
Suppose also that a critical point $\om_2$ of $P$ is (pre)periodic.
Then the other critical point $\om_1$ of $P$ is necessarily recurrent.
Either the boundary of the Siegel disk around $0$ or the point $0$ itself (if it is Cremer) is in the $\om$-limit set of $\om_1$.
Let $F$ be the planar fiber of $P$ containing $0$; set $Y=F\cap K_P$.
First note that $Y\ne K_P$ since $\om_2$ is by definition a valuable point.
Corollary \ref{c:main1} is applicable to $Y$.
Indeed, there are no parabolic cycles of $P$; otherwise the corresponding cycle of Fatou domains would contain $\om_2$.
By Corollary \ref{c:main1}, there is a quadratic polynomial $Q$ such that $Q|_{K_Q}$ is topologically conjugate to $P|_Y$.
There is an affine coordinate $z$ on $\C$, for which $Q$ has the form $Q(z)=\lambda z(z+1)$.
By Corollary \ref{c:pm}, we have $\lambda=P'(0)$.
Clearly, $0$ is Siegel (resp., Cremer) for $P$ if and only if it is Siegel (resp. Cremer) for $Q$.
The result now follows from the Theorem of Yoccoz \cite{yoc}.
\end{proof}

\subsection{Acknowledgements}
The authors are deeply grateful to the referee for many thoughtful suggestions.

\end{document}